\documentclass[reqno,centertags, 12pt]{amsart}
\usepackage{amsmath,amsthm,amscd,amssymb,mathtools}
\usepackage{latexsym}
\usepackage{graphicx}
\usepackage[dvipsnames]{xcolor}
\usepackage[mathscr]{eucal}
\def\corO{}
\def\corJ{}
\def\corJR{}
\def\corOR{}
\newcommand{\bbC}{{\mathbb{C}}}
\newcommand{\bbD}{{\mathbb{D}}}

\newcommand{\bbH}{{\mathbb{H}}}

\newcommand{\bbN}{{\mathbb{N}}}
\newcommand{\bbP}{{\mathbb{P}}}

\newcommand{\bbR}{{\mathbb{R}}}

\newcommand{\bbU}{{\mathbb{U}}}
\newcommand{\bbZ}{{\mathbb{Z}}}

\newcommand{\calC}{{\mathcal{C}}}

\newcommand{\calG}{{\mathcal G}}

\newcommand{\calI}{{\mathcal I}}

\newcommand{\calL}{{\mathcal L}}
\newcommand{\calM}{{\mathcal M}}
\newcommand{\calP}{{\mathcal P}}

\newcommand{\calT}{{\mathcal T}}

\newcommand{\calV}{{\mathcal V}}
\newcommand{\bdone}{{\boldsymbol{1}}}

\newcommand{\lb}{\label}

\newcommand{\ti}{\tilde  }
\newcommand{\wti}{\widetilde  }

\newcommand{\tr}{\text{\rm{Tr}}}

\newcommand{\ran}{\text{\rm{ran}}}

\newcommand{\bi}{\bibitem}

\newcommand{\beq}{\begin{equation}}
\newcommand{\eeq}{\end{equation}}
\newcommand{\ba}{\begin{align}}
\newcommand{\ea}{\end{align}}

\newcommand{\CUE}{\mbox{\rm CUE}}
\newcommand{\GUE}{\mbox{\rm GUE}}

\newcounter{smalllist}
\newenvironment{SL}{\begin{list}{{\rm\roman{smalllist})}}{%
\setlength{\topsep}{0mm}\setlength{\parsep}{0mm}\setlength{\itemsep}{0mm}%
\setlength{\labelwidth}{2em}\setlength{\leftmargin}{2em}\usecounter{smalllist}%
}}{\end{list}}

\newcommand{\comm}[1]{}

\DeclareMathOperator{\re}{Re}

\allowdisplaybreaks
\numberwithin{equation}{section}

\newtheorem{theorem}{Theorem}[section]

\newtheorem*{p2.1}{Proposition 2.1}
\newtheorem{proposition}[theorem]{Proposition}

\newtheorem{corollary}[theorem]{Corollary}
\theoremstyle{definition}
\newtheorem*{remark}{Remark}
\newtheorem*{remarks}{Remarks}

\newcommand{\jap}[1]{\langle #1 \rangle}
\newcommand{\norm}[1]{\lVert#1\rVert}
\newcommand{\overbar}{\overline}

\begin{document}

\title[Large Deviations and the Lukic Conjecture]{Large Deviations and the Lukic Conjecture}

\author[J.~Breuer, B.~Simon, and O.~Zeitouni]{Jonathan~Breuer$^{1,4}$, Barry~Simon$^{2,5}$,\\and
Ofer Zeitouni$^{3,6}$}

\thanks{$^1$ Institute of Mathematics, The Hebrew University, 91904 Jerusalem, Israel.
E-mail: jbreuer@math.huji.ac.il}

\thanks{$^2$ Departments of Mathematics and Physics, Mathematics 253-37, California Institute of Technology, Pasadena, CA 91125.
E-mail: bsimon@caltech.edu}

\thanks{$^3$ Faculty of Mathematics, Weizmann Institute of Science, POB 26, Rehovot 76100, Israel and Courant Institute, NYU.
E-mail: ofer.zeitouni@weizmann.ac.il.}

\thanks{$^4$ Research supported in part by \corJ{the Israel Science Foundation (Grant no.\ 399/16) and in part by the United States-Israel Binational Science Foundation (Grant No.\ 2014337)}.}

\thanks{$^5$  Research supported in part by NSF grant DMS-1265592 and in part by \corJ{the United States-Israel Binational Science Foundation (Grant No.\ 2014337)}.}

\thanks {$^6$ Research supported in part by a grant from the Israel Science Foundation.}

\

\date{\today}

\begin{abstract}  We use the large deviation approach to sum rules pioneered by Gamboa, Nagel and Rouault to prove higher order sum rules for orthogonal polynomials on the unit circle.  In particular, we prove one half of a conjectured sum rule of Lukic in the case of two singular points, one simple and one double.  This is important because it is known that the conjecture of Simon fails in exactly this case, so this paper provides support for the idea that Lukic's replacement for Simon's conjecture might be true.
\end{abstract}

\maketitle

\section{Introduction} \lb{s1}

This paper is a contribution to the theory of sum rules in the spectral theory of orthogonal polynomials.  The earliest such result is Szeg\H{o}'s Theorem for orthogonal polynomials on the unit circle (OPUC) in Verblunsky's form \cite{Verb} of which we'll say more soon.  The modern theory was initiated by Killip--Simon \cite{KS} for orthogonal polynomials on the real line (OPRL) with considerable work by others \cite{DK,GZ,Kupin,LNS,Lukic,Lukic2,NPVY,SZ}.

Here we'll consider OPUC. Given a probability measure $\mu$ on $\partial\bbD$, one can form the non--zero (in $L^2(\partial\bbD,d\mu)$), monic orthogonal polynomials $\{\Phi_n\}_{n=0}^M$ where $M=N-1$ if $\mu$ has exactly $N$ points in its support and $M=\infty$ if $\mu$ has infinitely many points in its support.  In the case there are exactly $N$ points, one defines $\Phi_N$ to be the unique degree $N$ monic polynomial \corO{vanishing} at all $N$ points (so $\Phi_N = 0$ in $L^2(\partial\bbD,d\mu)$).  The recursion (aka Verblunsky) coefficients, $\{\alpha_j\}_{j=0}^M$, are given by the recursion relations, $0 \le j < M+1$:
\begin{equation}\label{1.1}
  \Phi_{n+1}(z)=z\Phi_n(z) - \overbar{\alpha}_n \Phi^*_n(z); \quad \Phi_0 \equiv \bdone; \quad \Phi^*_n(z) = z^n \overbar{\Phi_n\left(\frac{1}{\bar{z}}\right)}
\end{equation}

For $N=\infty$, $\{\alpha_j\}_{j=0}^\infty \in \bbD^\infty$ (see \cite{OPUC1}) and for $N < \infty$, only $\alpha_0,\dots,\alpha_{N-1}$ are defined (since $\Phi_k$ is only defined for $k \le N$) and $\alpha_k \in \bbD, k=0,\dots,N-2, \, \alpha_{N-1} \in \partial\bbD$.

Verblunsky's Theorem states that there is a one--one correspondence, $\corJ{\calV}$, from probability measures to Verblunsky coefficients with the above restrictions, i.e. $\ran(\corJ{\calV} \restriction$ measures with infinite support) $= \prod_{j=0}^{\infty} \bbD$ and $\ran(\corJ{\calV}\restriction n$--point measures) $=\prod_{j=0}^{n-2} \bbD \times \partial\bbD$.  Moreover in the natural topologies, $\corJ{\calV}$ is a homeomorphism.

Szeg\H{o}'s Theorem in Verblunsky form says that
\begin{equation}\label{1.2}
  H\left(\left.\frac{d\theta}{2\pi}\right|\mu\right) = -\sum_{n=0}^{\corJR{M}} \log(1-|\alpha_n|^2)
\end{equation}
where $H(\nu|\mu)$ is the \corJ{Kullback-Leibler (KL)} divergence (aka $\pm$ \corJ{the} relative entropy, depending on the sign convention for the relative entropy)\corJ{:}
\begin{equation} \lb{1.3}
H(\nu \, | \,\mu) = \begin{cases}
  \int \log \left(\frac{d\nu}{d\mu}\right) \, d\nu , & \mbox{if } \nu \mbox{ is } \mu\mbox{--a.c.} \\
  \infty, & \mbox{otherwise}.
\end{cases}
\end{equation}

\eqref{1.2} always holds although both sides may be $+\infty$.
\corOR{(The latter is the case if, e.g.,} \corJR{ $M<\infty$ since in that case the $n=M$ term in the sum is $-\log(0)=\infty$).}
In particular, the condition that both sides are finite at the same time implies
\begin{equation}\label{1.4}
  \sum_{j=0}^{\infty} |\alpha_j|^2 < \infty \iff \int \log(w(\theta)) \frac{d\theta}{2\pi} > -\infty
\end{equation}
where
\begin{equation}\label{1.5}
  d\mu(\theta) = w(\theta)\frac{d\theta}{2\pi}+d\mu_s
\end{equation}
where $d\mu_s$ is singular w.r.t $d\theta$. Simon \cite{SzThm} calls a result like \eqref{1.4} that gives equivalence of spectral data and coefficient data a ``spectral theory gem''.  \eqref{1.4} in particular implies the existence of measures with arbitrarily bad singular part mixed in with a.c. spectrum and with $\ell^2$ decaying Verblunsky coefficients.

The current paper is devoted to higher order sum rules of which the first is that of Simon \cite[Section 2.8]{OPUC1}:
\begin{align}
  -\int (1-\cos \theta) \log(w(\theta)) \frac{d\theta}{2\pi} &= -\frac{1}{2}+\frac{1}{2}\sum_{n=-1}^{\infty}|\alpha_{n+1}-\alpha_n|^2 \nonumber \\
                                                            &\null \quad + \sum_{n=0}^{\infty}\left[-\log(1-|\alpha_n|^2)-|\alpha_n|^2\right] \lb{1.6}
\end{align}
where $\alpha_{-1} \equiv -1$.  This implies the gem
\begin{equation}\label{1.7}
  \corO{\int (1-\cos \theta) \log(w(\theta)) \frac{d\theta}{2\pi} > -\infty }\iff \sum_{n=0}^{\infty}|\alpha_{n+1}-\alpha_n|^2 + |\alpha_n|^4 \corO{<\infty}
\end{equation}

In the same section, Simon conjectured (wrongly as we'll see!) that for $\theta_1,\dots,\theta_k$ distinct in $[0,2\pi)$ and $m_1,\dots, m_k$ strictly positive integers we have that
\begin{equation}\label{1.8}
  \int \prod_{j=1}^{k} (1-\cos(\theta - \theta_j))^{m_j} \log(w(\theta)) \, d\theta > -\infty
\end{equation}
if and only if
\begin{flalign}\label{1.9}
  \textrm{(S1)} &&\mathclap{\prod_{j=1}^{k}(S-e^{i\theta_j})^{m_j} \alpha \in \ell^2}&&
\end{flalign}
and
\begin{flalign}\label{1.10}
  \textrm{(S2)} &&\mathclap{\alpha  \in \ell^{2m+2} \qquad m=\max_{j=1,\dots,k} m_j}&&
\end{flalign}
In \eqref{1.9}, $S$ is the operator
\begin{equation}\label{1.9A}
  (S\alpha)_n = \alpha_{n+1}
\end{equation}

Moreover, Simon--Zlato\v{s} \cite{SZ} proved this conjecture in case $\sum_{j=1}^{k} m_j = 2$, i.e. $k \le 2$ and $(m_1,m_2) = (2,0)$ or $(1,1)$.  For simplicity the remainder of this section will mainly discuss the case $\theta_1 =0, \, \theta_2=\pi$ although the next two sections will revert to the general case.  We'll use the symbol $(m_1,m_2)$ to describe this case.

In \cite{Lukic}, Lukic found a counterexample to this conjecture for the $(2,1)$ case.  He found an explicit example where $\textrm{(S1), (S2)}$ hold but
\begin{equation}\label{1.11}
  \int (1-\cos\theta)^2(1+\cos\theta) \log(w(\theta)) \, \frac{d\theta}{2\pi}
  = \corO{-\infty}
\end{equation}

To have any hope of an equivalence one needs more that $\textrm{(S1), (S2)}$.  Lukic made an improved conjecture that replaced $\textrm{(S1), (S2)}$ by
\begin{flalign}\label{1.12}
\textrm{(L}_1\textrm{1)}   &&\mathclap{\alpha \textrm{ can be written } \alpha = \beta^{(1)}+\cdots+\beta^{(k)}}&&
\end{flalign}
\begin{flalign}\label{1.13}
   \textrm{(L}_1\textrm{2)} &&\mathclap{(S-e^{i\theta_j})^{m_j}\beta^{(j)} \in \ell^2}&&
\end{flalign}
\begin{flalign}\label{1.14}
   \textrm{(L}_1\textrm{3)} &&\mathclap{\beta^{(j)} \in \ell^{2m_j+2}}&&
\end{flalign}

Lukic also proved a flawed gem, i.e. an equivalence under an a priori condition on the Verblunsky coefficients, that provides evidence for his conjecture.  In Section \ref{s7} we'll obtain some additional evidence for the correctness of the Lukic conjecture.  In Section \ref{s2}, we'll consider equivalent versions of Lukic's conditions that are directly expressible in terms of $\alpha$ without reference to a decomposition as a sum.  In a sense, $\beta^{(j)}$ is the part of $\alpha$ localized near $\theta_j$ in Fourier space, so that for the $(2,1)$ case Lukic's conditions are equivalent to (the $\textrm{(L}_2\textrm{)}$ conditions will appear in the next section).
\begin{flalign}\label{1.15}
\textrm{(L}_3\textrm{1)}   &&\mathclap{(S-1)^2(S+1)\alpha \in \ell^2}&&
\end{flalign}
\begin{flalign}\label{1.16}
\textrm{(L}_3\textrm{2)}   &&\mathclap{(S-1)^2\alpha \in \ell^4}&&
\end{flalign}
\begin{flalign}\label{1.17}
\textrm{(L}_3\textrm{3)}   &&\mathclap{\alpha \in \ell^6}&&
\end{flalign}
In this case $\textrm{(L}_3\textrm{1)}\equiv \textrm{(S1)}$, $\textrm{(L}_3\textrm{3)} \equiv \textrm{(S2)}$ and $\textrm{(L}_3\textrm{2)}$ is an extra condition.  The precise result we'll prove in Section \ref{s7} is that the $\textrm{(L}_3)$ conditions imply the finiteness of the integral in \eqref{1.11} (at least when the $\alpha$s are real).

Recently, Gamboa, Nagel and Rouault \cite{GNR1} (henceforth GNR; see also \cite{GNR2,GNR3}) discovered a new approach to Szeg\H{o}'s Theorem (and the Killip--Simon Theorem) using the theory of large deviations (LD).  We wrote a pedagogical presentation of some of these ideas \cite{BSZ1}.  Our main goal in this paper is to use large deviation methods to study higher order sum rules.  We note that GNR \cite{GNR3} discussed \eqref{1.6} using LD methods although for technical reasons, they were unable to prove the actual sum rules.  Below we will assume the reader familiar with some of the basics of LD theory either from books \cite{DS,DZ} or from our paper \cite{BSZ1}.

In Section \ref{s3}, we will prove a sum rule and gem where one side of the gem is the integral in \eqref{1.8}.  In general, the other side of the gem will be a very complicated polynomial in the $\alpha$'s (with some non--polynomial terms of the form $\log(1-|\alpha|^2)$).  This leads to a new insight.  The Lukic conjecture (if true) provides much more humane conditions on the $\alpha$'s than what one gets from the naive sum rule.  We note that we suspect that our sum rules are identical to the ones found by Denissov--Kupin \cite{DK} who did not carry through the examples of Sections \ref{s4}-\ref{s7}.  Section \ref{s4} will use these ideas to get the sum rule \eqref{1.6} in a new way.

In the last four sections, we make two simplifying assumptions:\\

\begin{SL}
\item[(A1)] when $k=2, \, \theta_1 = 0, \, \theta_2 = \pi$ (essentially $\theta_2-\theta_1 = \pi$ is what is important) and we'll also consider the symmetric situation where one has general $k$ points symmetrically arranged as the roots of unity with all $m_j=1$.
\item[(A2)] $\bar{\alpha}_j = \alpha_j$ for all $j \ge 0$.
\end{SL}

\null

These are mainly to make the sometimes involved calculations simpler.  We have no doubt that one can do the calculations without \textrm{(A2)} and suspect one can drop \textrm{(A1)} although with some effort.

In Sections \ref{s5} and \ref{s6}, we recover the Simon--Zlato\v{s} gems (i.e. $(1,1)$ and $(2,0)$), at least under the assumptions \textrm{(A1)--(A2)}. One thing we'll see in these sections is that it is simpler to show that the conditions on Verblunsky coefficients imply the measure condition than the converse so in the last two sections, we'll settle for the simpler half.  In Section \ref{s6.5}, we'll prove this one direction for $k$ equally spaced points, all with order 1, that is we'll prove that $\sum_{n=1}^{\infty} |\alpha_{n+k}-\alpha_n|^2+|\alpha_n|^4 < \infty \Rightarrow \int (1-\cos k\theta) \log(w(\theta))\, d\theta > -\infty$ and in Section \ref{s6.6}, we discuss an arbitrarily high order single singular point under the hypothesis that $\alpha \in \ell^4$.  The results in Sections \ref{s6.5} and \ref{s6.6} are not new but recover, using new methods, special cases of results of Golinskii--Zlato\v{s} \cite{GZ}.  In Section \ref{s7}, we will prove that for the  $(2,1)$ case under \textrm{(A1)/(A2)}, the Lukic conditions imply finiteness of the integral.  Recall that this is a case where the Simon conditions \corJ{do} not imply finiteness of the integral so we regard this as strong evidence for the Lukic conjecture.

We believe our main results in this paper are the general sum rule and gem and the realization that the Lukic conjecture is just about finding a simpler version of the naive Verblunsky coefficient side.  In addition we show how to use LD methods to recover the gems of Simon and Simon--Zlato\v{s} and some  results of Golinskii--Zlato\v{s}.  Finally, we provide evidence for the general Lukic conjecture by finding a situation where his conditions imply the finiteness of the relevant integral and where Simon's do not.\footnote{\corOR{Building on the general sum rules of this paper, Jun Yan developed an algebraic machinery that allowed him
to obtain new examples where (one half of) the Lukic conjecture can be verified. We refer to \cite{JY} for details.}}

We thank Peter Yuditskii for telling two of us about \cite{GNR1} and
\corO{Fabrice} Gamboa, \corO{Jan} Nagel and \corO{Alain}
Rouault for useful discussions.

\section{The Lukic Condition} \lb{s2}

In this section, we want to discuss some equivalent forms of the Lukic conditions $\textrm{(L}_1\textrm{1-3)}$.  This and some of the analysis in later sections will require some discrete hard analysis that we set up here.  First, we'll consider

\begin{flalign}\label{2.1}
\textrm{(L}_2\textrm{1)}   &&\mathclap{\prod_{j=1}^{k}(S-e^{i\theta_j})^{m_j}\alpha \in \ell^2}&&
\end{flalign}

\begin{flalign}\label{2.2}
\textrm{(L}_2\textrm{2)}   &&\mathclap{\prod_{j \ne q}(S-e^{i\theta_j})^{m_j}\alpha \in \ell^{2m_q+2}, \qquad q=1,\dots,k}&&
\end{flalign}

In some sense, $\textrm{(L}_2\textrm{2)}$ says that in ``$\theta$--space'' $\alpha$ is locally $\ell^{2m_q+2}$ near $\theta=\theta_q$.  Our first result is

\begin{theorem}\lb{T2.1} $\textrm{(L}_1\textrm{1-3)} \iff \textrm{(L}_2\textrm{1-2)}$
\end{theorem}

\begin{remark} The same argument shows that $\textrm{(S}\textrm{1-2)}$ are equivalent to \eqref{2.1} and \eqref{2.2} but with $2m_q+2$ replaced by $2\max m_j + 2$.  This illustrates the difference between the Simon and Lukic conditions.
\end{remark}

The proof will depend on momentum space localization.  We can view $\ell^q(\bbN)$ as a subspace of $\ell^q(\bbZ)$ and define $P:\ell^q(\bbZ) \to \ell^q(\bbN)$ by restricting $\{a_n\}_{n=-\infty}^\infty$ to $\{a_n\}_{n=0}^\infty$.  We can think of $P$ either as a map between spaces which clearly has norm $1$ or as a map of $\ell^q(\bbZ)$ to itself whose range is those $a$ with $a_n = 0$ for all $n < 0$.  In the latter view, $P$ is a projection of norm $1$.  We can extend $S$ to $\ell^q(\bbZ)$ by setting $(Sa)_n = a_{n+1}$.  This $S$ is an invertible isometry (on $\ell^q(\bbN)$ it doesn't have a \corJ{left} inverse).

$S$ is unitary on $\ell^2(\bbZ)$ with spectrum all of $\partial\bbD$, so, by the spectral theorem, we can define $F(S)$ on $\ell^2(\bbZ)$ for any $F \in L^\infty(\bbD)$ and then $F_+(S)$ on $\ell^2(\bbN)$ by $a \mapsto PF(S)a$.  These are sometimes called Laurent and Toeplitz operators respectively.  $F(S)$ is made most transparent by using Fourier transform, $f \mapsto f^\#$, mapping $L^2(\partial\bbD,\tfrac{d\theta}{2\pi})$ to $\ell^2(\bbZ)$ by
\begin{equation}\label{2.3}
  f^\#_k = \int_{0}^{2\pi} e^{-ik\theta} f(e^{i\theta}) \frac{d\theta}{2\pi}
\end{equation}
These are, of course, Fourier coefficients of $f$ in the orthonormal basis of $L^2(\partial\bbD,\tfrac{d\theta}{2\pi})$, $\{e^{ik\theta}\}_{k=-\infty}^\infty$, so we can define $a \mapsto a^\flat$ from sequences to functions by defining (with convergence in $L^2$--sense):
\begin{equation}\label{2.4}
  a^\flat(e^{i\theta}) = \sum_{k=-\infty}^{\infty} a_k e^{ik\theta}
\end{equation}
Then $(f^\#)^\flat=f$.

If $F$ is a trigonometric polynomial so
\begin{equation*}
  F(e^{i\theta}) = \sum_{k=-M}^{M} F^\#_k e^{ik\theta}
\end{equation*}
then, for $f \in L^2$,
\begin{align}
  (Ff)^\#_k &= \int_{0}^{2\pi} e^{-ik\theta} \left[\sum_{j=-M}^{M} F^\#_j e^{ij\theta}\right] f(e^{i\theta}) \, \frac{d\theta}{2\pi} \nonumber \\
            &= \sum_{j=-M}^{M} F^\#_j f^\#_{k-j} \lb{2.5}
\end{align}
i.e.\ $F(S)$ is convolution with $F^\#$.  If $F \in C^\infty(\partial\bbD)$, by a simple argument (see \cite[Section 6.3]{RA}), $F_k^\#$ decays faster than any inverse polynomial, so, in particular, $F^\# \in \ell^1$.  Taking limits in \eqref{2.5}, we see that formula still holds but with $M$ replaced by $\infty$.  Thus, since $F^\# \in \ell^1$, we see that as maps on $\ell^p(\bbZ)$ or $\ell^p(\bbN)$, $a \mapsto F(S)a$ maps $\ell^p$ to itself, and since $P$ maps
$\ell^p\corO{(\bbN)}$
to itself, we see that $F_+(S)$ map $\ell^p(\bbN)$ to itself i.e.\

\begin{proposition} \lb{P2.2} $a\mapsto F(S)a$ maps any $\ell^p(\bbZ)$ to itself and $a \mapsto F_+(S)a$ maps any $\ell^p(\bbN)$  for $1 \le p < \infty$ for any $C^\infty$ function, $F$, on $\partial\bbD$.
\end{proposition}

In particular, we can localize in $\theta$--space by picking a convenient partition of unity on $\partial\bbD$ and writing $a=\sum_{j=1}^{k} J_j(S)a$.

\begin{corollary} \lb{C2.3}  Let $Q(z)$ be a Laurent polynomial on $\bbC \setminus \{0\}$.  Let $F$ be a $C^\infty$ function on $\partial\bbD$ so that $Q(z)$ has no zeros in the support of $F$.  Suppose that $a$ lies in some $\ell^q$.  Let $1 \le p < \infty$. Then
\begin{align}
  Q(S)F(S)a \in \ell^p(\bbZ) &\Rightarrow F(S)a \in \ell^p(\bbZ) \lb{2.6} \\
  Q_+(S)F_+(S)a \in \ell^p(\bbN) &\Rightarrow F_+(S)a \in \ell^p(\bbN) \lb{2.6a}
\end{align}
\end{corollary}

\begin{proof} Suppose first we are dealing with the maps on $\ell^q(\bbZ)$.  By the zero condition, it is easy to find a $C^\infty$ function, $G$, on $\partial\bbD$ so that $G(z)Q(z)F(z) = F(z)$ for all $z \in \partial\bbD$.  Thus, if $Q(S)F(S)a \in \ell^p$, then
\begin{equation*}
  F(S)a = G(S)Q(S)F(S)a \in \ell^p
\end{equation*}
since $G(S)$ maps $\ell^p$ to $\ell^p$.

Now suppose \corJ{$(a_n) \in \ell^q(\bbN)$ and extend it to $\bbZ$ by $a_n = 0$} for $n <0$.  Since $F(S)$ is convolution with a function of very rapid decay, $F(S)a$ and $Q(S)F(S)a$ both have rapid decay to the left so since $PQ(S)PF(S)a$ lies in $\ell^{\corJ{p}}(\corJ{\bbN})$, we see that $Q(S)PF(S)a$ lies in $\ell^p(\corJ{\bbZ})$.  Since $F(S)a$ has rapid decay on the left, $Q(S)(1-P)F(S)a$ lies in $\ell^p(\corJ{\bbZ})$ and so $Q(S)F(S)a$ lies in $\ell^p(\corJ{\bbZ})$ as well.  By the argument in the first paragraph, $F(S)a$ lies in $\ell^p(\corJ{\bbZ})$ so $PF(S)a$ lies in $\ell^p(\corJ{\bbN})$.
\end{proof}

\begin{proof} [Proof of Theorem \ref{T2.1}] ($\textrm{L}_2 \Rightarrow \textrm{L}_1$) Let $\corO{\alpha}$ obey $\textrm{L}_2$.  Pick $\{J_j\}_{j=1}^k$, $C^\infty$ functions on $\partial\bbD$ so that $J_j \ge 0, \sum_{j=1}^{k} J_j = 1$ and $J_j$ vanishes in the neighborhood of $\{\theta_\ell\}_{\ell \ne j}$.  Let $\beta^{(j)} = PJ_j(S)\alpha$.  $\textrm{(L}_1\textrm{1)}$ follows from $\sum_{j=1}^{k} J_j = 1$.  Since $J_j(S)$ commutes with any polynomial in $S$, by \eqref{2.1},
\begin{equation}\label{2.7}
  \left[\prod_{j=1,j \ne q}^{k}(S-e^{i\theta_j})^{m_j}\right](S-e^{i\theta_q})^{m_q}\beta^{(q)} \in \ell^2
\end{equation}
(with a small argument to deal with the P operator) so, by Corollary \ref{C2.3}, \eqref{1.13} holds.  A similar argument shows that \eqref{2.2} implies \eqref{1.14}.

($\textrm{L}_1 \Rightarrow \textrm{L}_2$) Suppose $\corO{\alpha}$ obeys $\textrm{L}_1$.  Since polynomials in $S$ map $\ell^p$ to itself, \eqref{1.13}$\Rightarrow\prod_{j=1}^{k}(S-e^{i\theta_j})^{m_j} \beta^{(q)} \in \ell^2$, so by
\corO{\eqref{1.12}}, we have \eqref{2.1}.  By \eqref{1.13}, if $r \ne q$, then $\prod_{j \ne q}(S-e^{i\theta_j})^{m_j}\beta^{(r)} \in \ell^2 \subset \ell^{2m_q+2}$.  Also \eqref{1.14} implies $\prod_{j \ne q}(S-e^{i\theta_j})^{m_j}\beta^{(q)} \in \ell^{2m_q+2}$.  Therefore, by \eqref{1.12}, we get \eqref{2.2}.
\end{proof}

For comparison with Simon's conjecture, the following version (which appeared already in the last section) is useful.  Let $m=\sup_j m_j$,
\begin{flalign}\label{2.8}
\textrm{(L}_3\textrm{1)}   &&\mathclap{\prod_{j=1}^{k}(S-e^{i\theta_j})^{m_j}\alpha \in \ell^2}&&
\end{flalign}
\begin{flalign}\label{2.9}
\textrm{(L}_3\textrm{2)}   &&\mathclap{\textrm{For } m_q<m \qquad \prod_{j \ne q}(S-e^{i\theta_j})^{m_j}\alpha \in \ell^{2m_q+2}}&&
\end{flalign}
\begin{flalign}\label{2.10}
\textrm{(L}_3\textrm{3)}   &&\mathclap{\alpha \in \ell^{2m+2}}&&
\end{flalign}

\begin{theorem} \lb{T2.4} $\textrm{(L}_1\textrm{1-3)} \iff \textrm{(L}_3\textrm{1-3)}$
\end{theorem}

\begin{proof} Clearly, \eqref{2.10} implies \eqref{2.2} when $m=m_j$, so $\textrm{(L}_3\textrm{1-3)} \Rightarrow \textrm{(L}_2\textrm{1-2)} \Rightarrow \textrm{(L}_1\textrm{1-3)}$.

  On the other hand, by Theorem \ref{T2.1}, $\textrm{(L}_1\textrm{1-3)} \Rightarrow \textrm{(L}_3\textrm{1-\corOR{2})}$ and trivially, \eqref{1.12} and
  \corO{\eqref{1.14}} $\Rightarrow$ \eqref{2.10}
\end{proof}

To find some equivalent forms of the Lukic conditions, it will be useful to have the following:

\begin{theorem} \lb{T2.5} For any sequence $\corO{\alpha} \in \ell^2(\bbZ)$ of finite support, we have that:
\begin{equation}\label{2.11}
  \norm{(S-1)\corO{\alpha}}_3^2 \le 2 \norm{(S-1)^2\corO{\alpha}}_2 \norm{\corO{\alpha}}_6
\end{equation}
\end{theorem}

\begin{remarks} 1. This is a discrete case of an inequality on derivatives due to Gagliardo \cite{Gag} and Nirenberg \cite{Niren}; see Simon \cite[Section 6.3]{HA} and Taylor \cite{Tay}.  Here $S-1$ replaces $\tfrac{d}{dx}$.  The general version (with essentially the same proof) is
\begin{equation*}
  \norm{(S-1)\corO{\alpha}}^2_{2k/p} \le \frac{2k-p}{p} \norm{(S-1)^2\corO{\alpha}}_{2k/(p+1)} \norm{\corO{\alpha}}_{2k/(p-1)}
\end{equation*}
for $k \ge 1, 1\le p\le k$. \eqref{2.11} is $p=2, k=3$.

2. Once one has \corO{Theorem \ref{T2.5}} then
it is easy to show, by dominated convergence,
that $\corO{\alpha} \in \ell^6, \, (S-1)^2\corO{\alpha} \in \ell^2 \Rightarrow (S-1)\corO{\alpha} \in \ell^3$ and \corO{that}
\eqref{2.11} holds
\corO{even without the condition on finite support of $\corO{\alpha}$}.

3.  This result is in \cite{SZ} and probably other places but the proof is so simple that we give it for the reader's convenience.

4.  \eqref{2.17} below can be thought of as resulting from a summation by parts.
\end{remarks}

\begin{proof} Given $\corO{\alpha}$, define $|\corO{\alpha}|$ by $|\alpha|_n \equiv |\alpha_n|$.  We begin by noting that for $a,b \in \bbC$, we have by the triangle inequality that
\begin{equation}\label{2.12}
  ||a|-|b|| \le |a-b|
\end{equation}
so that if $\corO{\alpha} \le \corO{\beta} \iff \alpha_n \le \beta_n$ for all $n$, then
\begin{equation}\label{2.13}
  |(S-1)|\corO{\alpha}|| \le |(S-1)\corO{\alpha}|
\end{equation}

Note next that Leibniz rule takes the form (where $(\corO{\alpha}\corO{\beta)}_n = \alpha_n\beta_n$)
\begin{equation}\label{2.14}
  (S-1)(\corO{\beta}\corO{\gamma}) = [(S-1)\corO{\beta}]\corO{\gamma}+(S\corO{\beta})[(S-1)\corO{\gamma}]
\end{equation}
so
\begin{align}
  (S-1)(\corO{\beta}\corO{\gamma}\corO{\kappa}) &= [(S-1)\corO{\beta}]\corO{\gamma}\corO{\kappa} + (S\corO{\beta})[(S-1)\corO{\gamma}]\corO{\kappa} \nonumber \\
                                                                    &\null \qquad \quad + (S\corO{\beta})(S\corO{\gamma})[(S-1)\corO{\kappa}] \lb{2.15}
\end{align}
Choose $\corO{\beta}=\corO{\alpha},\, \corO{\gamma}=(S-1)\corO{\bar{\alpha}}, \, \corO{\kappa}=|(S-1)\corO{\alpha}|$ and use the fact that a sum of $(S-1)\corO{\tau}$ is zero when $\corO{\tau}$ has finite support (because of telescoping) to see that if $\alpha$ has finite support, then
\begin{align}\label{2.17}
  \sum_n |[(S-1) \alpha]_n|^3 &\le \sum_n |S\alpha|_n
  \corO{|(S-1)^2\bar{\alpha}|_n
|(S-1)\alpha|_n}
  \nonumber \\
                              & \null \qquad \quad + \sum_n |S\alpha|_n |S(S-1)\bar{\alpha}|_n |(S-1)^2\alpha|_n
\end{align}
where we used \eqref{2.13} to bound $|(S-1)|(S-1)\corO{\alpha}||$ by $|(S-1)^2\corO{\alpha}|$.

H\"{o}lder's inequality and $\tfrac{1}{6} + \tfrac{1}{3} + \tfrac{1}{2}=1$ says that the first sum on the right is bounded by $\norm{S\alpha}_6 \norm{(S-1)\alpha}_3 \norm{(S-1)^2 \bar{\alpha}}_2 = \norm{\alpha}_6 \norm{(S-1)\alpha}_3\norm{(S-1)^{\corO{2}}\alpha}_2$.  The second sum has the same bound which shows that
\begin{equation*}
  \norm{(S-1)\alpha}_3^3 \le 2 \norm{\alpha}_6 \norm{(S-1)\alpha}_3\norm{(S-1)^{\corO{2}}
  \alpha}_2
\end{equation*}
which implies \eqref{2.11}
\end{proof}

Let us focus on the case $(\theta_1,\theta_2,m_1,m_2)=(0,\pi,2,1)$, so we have
\begin{flalign}\label{2.18}
\textrm{(L}_3\textrm{1)}   &&\mathclap{(S+1)(S-1)^2\alpha \in \ell^2}&&
\end{flalign}
\begin{flalign}\label{2.19}
\textrm{(L}_3\textrm{2)}   &&\mathclap{(S-1)^2\alpha \in \ell^{4}}&&
\end{flalign}
\begin{flalign}\label{2.20}
\textrm{(L}_3\textrm{3)}   &&\mathclap{\alpha \in \ell^{6}}&&
\end{flalign}
We want to note that

\begin{theorem} \lb{T2.6}  $\textrm{(L}_3\textrm{1-3)}$ for $(\theta_1,\theta_2,m_1,m_2)=(0,\pi,2,1)$ is equivalent to
\begin{flalign}\label{2.21}
\textrm{(L}_4\textrm{1)}   &&\mathclap{(S+1)(S-1)^2\alpha \in \ell^2}&&
\end{flalign}
\begin{flalign}\label{2.22}
\textrm{(L}_4\textrm{2)}   &&\mathclap{(S-1)\alpha \in \ell^{4}}&&
\end{flalign}
\begin{flalign}\label{2.23}
\textrm{(L}_4\textrm{3)}   &&\mathclap{\alpha \in \ell^{6}}&&
\end{flalign}
Moreover, one also has that if these conditions hold, then
\begin{equation}\label{2.24}
  (S^2-1)\alpha \in \ell^3
\end{equation}
\end{theorem}

\begin{remarks} 1. The proof shows that when $\textrm{(L}_3\textrm{1)}$ and $\textrm{(L}_3\textrm{3)}$ hold, then $(S-1)\alpha \in \ell^4$ is equivalent to $(S-1)^{k+1}\alpha \in \ell^4$ for any k fixed $k=1,2,\dots$.

2. The example $\alpha_n=(n+1)^{-1/5}$ obeys $\textrm{(L}_4\textrm{1-3)}$ but doesn't have $\alpha \in \ell^4$.
\end{remarks}

\begin{proof} Clearly $\textrm{(L}_4\textrm{1-3)} \Rightarrow \textrm{(L}_3\textrm{1-3)}$ since $S-1$ maps $\ell^4$ to itself.  So suppose we have $\textrm{(L}_3\textrm{1-3)}$.  Applying \eqref{2.11} to $(S+1)\alpha$ and noting that $\alpha \in \ell^6 \Rightarrow (S+1)\alpha \in \ell^6$, we conclude that $(S-1)(S+1)\alpha = (S^2-1)\alpha \in \ell^3$ proving \eqref{2.24}.

Since $p > q \Rightarrow \ell^q \subset \ell^p$, we see that $(S^2-1)\alpha \in \ell^4$.  Thus
\begin{equation}\label{2.25}
  (S^2-1)\alpha -(S-1)^2\alpha = 2(S-1)\alpha \in \ell^4
\end{equation}
\end{proof}

\section{Sum Rules} \lb{s3}

In this section, we'll explain how to use LD methods to obtain sum rules for any choice of $\{m_j\}_{j=1}^k$ and $\{\theta_j\}_{j=1}^k$ where one side is \eqref{1.8}.  The sum rules imply gems.  In fact, it will be easier to obtain the gems and we'll prove them first as part of the proof of sum rules.  While we haven't tried to prove it in general, we believe our sum rules are the same as those of Denisov--Kupin \cite{DK} obtained using the method of Nazarov et. al. \cite{NPVY}.

We begin by finding matrix models whose LDP on the spectral side involves \eqref{1.8} up to constants.  Our basic random matrix measures will have the form
\begin{equation}\label{3.1}
  Z_N^{-1} e^{-NQ(U)} \, d\bbH_N(U)
\end{equation}
where $Z_N$ is a normalization factor, $Q$ is a function of $U$ of the form
\begin{equation}\label{3.2}
  Q(U) = \tr(V(U))
\end{equation}
where $V$ is a Laurent polynomial
\begin{equation}\label{3.3}
  V(z) = \sum_{\ell=-k}^{k} c_\ell z^\ell
\end{equation}
(if $c_k \ne 0$ and/or $c_{-k} \ne 0$, we say that $k$ is the degree of $Q$ or $V$) and where $\bbH_N$ is Haar measure (aka circular unitary ensemble, $\CUE(n)$).  GNR \cite{GNR1, GNR3} also discussed these models, especially the case $V(e^{i\theta}) = \cos \theta$ (discussed first, \corJ{in a different context}, by Gross--Witten \cite{GW} whose name GNR assign to the model) but they do not prove sum rules or gems for these models.

There is a huge literature on these matrix models, discussed for example in \cite[Section 2.7]{AGZ}.  Much of the literature discusses perturbations of $\GUE$ rather than $\CUE$ but the results that we need extend to $\CUE$, which is technically simpler because random unitary matrices, unlike random self--adjoint matrices, are automatically uniformly bounded.  A major result (see, for example, \cite[Section 2.6]{AGZ}) is that the associated limit of empirical measures (aka density of states), $d\eta$, obeys
\begin{equation} \lb{3.4}
V(e^{i\theta}) = 2 \int \log(|e^{i\theta}-e^{i\psi}|) d\eta(\psi) + C
\end{equation}
for some constant $C$ (which when we start with $\eta$ we will take to be zero).

Any fixed vector, $\varphi \in \bbC^n$, is a cyclic vector for a.e.\ $U \in \bbU_n$.  Associated to each such $U$ is a probability measure $\mu$ on $\partial\bbD$ which is an $n$--point measure with masses at the eigenvalues of $U$ and weights the absolute square of the components of $\varphi$ in the corresponding eigenvectors.  Thus picking $\varphi$ (conventionally to be $\delta_1 = (1,0,\dots,0)$), we get a many-to-one correspondence between a set of unitaries of full measure and all $n$-point spectral measures.  Thus the measure in \eqref{3.1} induces a probability measure on $n$-point probability measures and so on sets of Verblunsky coefficients.  The unitaries $U$ and $U'$ correspond to the
same spectral measure if and \corO{only} if there is a unitary $\corJ{W}$ which has $\varphi$ as an eigenvector with $U'=\corJ{W}U\corJ{W}^{-1}$.  It is important to notice that the spectral measure determines the eigenvalues of $U$ and so $\tr(U^k)$ for any k,  so these traces are only functions of the Verblunsky coefficients and we can compute the traces in any convenient representation of one of the unitaries associated to a given spectral measure.

The measure in \eqref{3.1} induces a measure $\bbP_N$ on $N$--point measures
\corOR{(the spectral measures, viewed as elements of $\calM_{+,1}(\partial \bbD)$)},
and the Verblunsky map drags that to a measure $\wti{\bbP}_N$ on the set of $N$-point Verblunsky coefficients, i.e. $\bbD^{N-1} \times \partial\bbD$.

\corOR{The measure in \eqref{3.1}  induces another measure on the sequence of
empirical measures $L_N=\frac{1}{N}\sum_{i=1}^N \delta_{\lambda_i}\in \calM_{+,1}(\partial \bbD)$, where
$\lambda_i$ are the eigenvalues of $U$.} Recall that if  $V$ obeys \eqref{3.4}, then $L_N$ converges \corJR{a.s.\ } as $N \to \infty$ to \corO{$\eta$} and by the method of Ben Arous--Guionnet \cite{BAG}, the sequence $L_N$ obeys a LDP \corOR{(in the usual topology of weak convergence of probability measures)} with speed \corO{$N^2$} and rate function at measure $\mu$, $E(\mu) - E(\eta)$ where $E$ is the 2D Coulomb energy in external field which is minimized at $\mu = \eta$ (by \eqref{3.4}).

By the arguments in \cite[Section 3]{BSZ1}, if \corOR{the support of $\eta$ is all of $\partial \bbD$ and $\eta$ possesses a density with respect
to Lebesgue's measure which is strictly positive $d\theta$-almost everywhere},
one finds that the spectral measure obeys an LDP in
\corOR{$\calM_{+,1}(\partial \bbD)$}
with speed $N$ and rate function
\begin{equation}\label{3.5}
  I(\mu) = H(\eta\,|\,\mu)
\end{equation}
where $H$ is given by \eqref{1.3}. \corOR{On the other hand, as discussed
in \cite{GNR1} and \cite{BSZ1}, by the continuity
of the map $\calV$, the latter LDP induces  a LDP on the infinite sequence
of Verblunsky coefficients, viewed as elements of $\bbD^{\bbZ_+}$ equipped
with the product topology, with rate function given in terms of $I$. By  the uniqueness of the rate functions in large deviations
theory,} \corJR{if one has an expression for the rate function in terms of the Verblunsky coefficients then one 
gets} a sum rule with the integral in \eqref{1.8} on one side (up to constants due to the normalization of $\eta$ and a $\int \log\left(\frac{d\eta}{d\theta}\right) \, d\eta(\theta)$) term.

\corOR{We remark that the regularity assumptions} {\corJR{stated above for $\eta$ (namely full support and a.e.\ positive density) make it possible} \corOR{to mimic the proof in \cite[Section 3]{BSZ1} and
approximate the spectral measure throughout its support; to see
what goes wrong when there are gaps in the support of $\eta$, it is enough to consider the analogous problem for Hermitian matrices where $\partial \bbD$ is replaced by $\bbR$. In that case, there may be
``stray eigenvalues'' which are not controlled by the LDP for the
empirical measure. We refer to
\cite{GNR1} for a discussion of this issue, and \cite{GNR3} for
a detailed proof of the LDP for the spectral measure in the cases treated in this
paper.}

In what follows, we will be interested in $\eta$ of the form
\begin{equation}\label{3.6}
  \eta = Z_\eta^{-1} \prod_{j=1}^{k} [1-\cos(\theta - \theta_j)]^{m_j} \frac{d\theta}{2\pi},
\end{equation}
\corOR{which automatically satisfies the regularity assumption stated above. }

In computing \eqref{3.4} with that $\eta$, the following is useful
\begin{proposition} \lb{P3.1} For any $n \in \bbZ$, $n \ne 0$, we have that
\begin{equation}\label{3.7}
  -\corO{\int_0^{2\pi}} e^{in\psi} \log |e^{i\psi} - e^{i\theta}| \, \frac{d\psi}{2\pi} = \frac{e^{in\theta}}{2|n|}
\end{equation}
If $n=0$, the integral is zero.
\end{proposition}

\begin{proof}  While this integral is in the tables, the proof is so simple we give it. Replacing $\psi$ by $\psi-\theta$, we can suppose that $\theta = 0$.  By taking complex conjugates, we can suppose that $n \le 0$.  Write $e^{i\psi}=z$ and
\begin{equation*}
  \log|z-1| = \frac{1}{2}\log|z-1|^2 = \frac{1}{2}(\log(1-z)+\log(1-\bar{z}))
\end{equation*}
Then note that for $n < 0$
\begin{equation*}
  \corO{\oint_{\corJ{\partial{\bbD}}}}
  \bar{z}^{-n} \log(1-\bar{z}) \, \frac{d\bar{z}}{\bar{z}} = \overbar{\corO{
    \oint_{\corJ{\partial{\bbD}}}} z^{-n-1} \log(1-z) \, dz} = 0
\end{equation*}
by the Cauchy integral theorem.  By the Cauchy formula for Taylor coefficients and the well known series $\log(1-z) = -\sum_{n=1}^{\infty} \tfrac{z^n}{n}$, for $n \le 0$ (since the series only converges inside the disk, one needs to note that the integral over the unit circle is a limit of integrals over slightly smaller circles)
\begin{equation}\label{3.8}
  \frac{1}{2\pi i} \corO{\oint_{\corJ{\partial{\bbD}}}} z^{n-1} \log(1-z) \, dz = \begin{cases}
                                                    0, & \mbox{if } n=0 \\
                                                    -\frac{1}{|n|}, & \mbox{if }n<0
                                                  \end{cases}
\end{equation}
\end{proof}
Thus, for $\eta$ of the form \eqref{3.6}, $V$ defined by \eqref{3.4} is a Laurent polynomial with no constant term.

As a preliminary to the calculation of the Verblunsky coefficient side, we want to make two comments about the \corJ{sum rules and their relation to the rate function. The first one regards the fact that rather than the integral in \eqref{1.8}, the form of the rate function on the measure side is $H(\eta|\mu)$, which involves an additional term of the form
\begin{equation} \nonumber
\int \log \left( \prod_{j=1}^k [1-\cos(\theta-\theta_j)]^{m_j} \right) \prod_{j=1}^k [1-\cos(\theta-\theta_j)]^{m_j} \frac{d\theta}{2\pi}
\end{equation}
Computing this constant term is important in writing the sum rule. As an example,} rather than the left side of \eqref{1.6}, the LD calculation will give $H(\eta|\mu)$ where
\begin{equation}\label{3.9}
  d\eta(\theta) = (1-\cos \theta)\, \frac{d\theta}{2\pi}
\end{equation}
Noting that $\int (1-\cos \theta) \log(1-\cos \theta)\,\tfrac{d\theta}{2\pi} = 1-\log(2)$ (which follows as in the proof of Proposition \ref{P3.1}; see \cite[Section 2.8]{OPUC1}) we can write \eqref{1.6} as
\begin{align}
   H(\eta|\mu) &= 1-\log(2) + \frac{1}{2} |\alpha_0|^2 + \re(\alpha_0) + \frac{1}{2}\sum_{n=0}^{\infty} |\alpha_{n+1}-\alpha_n|^2 \nonumber \\
               & \null \qquad\qquad + \sum_{n=0}^{\infty}\left[-\log(1-|\alpha_n|^2)-|\alpha_n|^2\right] \lb{3.10}
\end{align}

The \corJ{right hand side has to vanish} when the $\alpha_n$ are the Verblunsky coefficients of the measure $\eta$ (since $H(\eta|\eta)=0$).  Let us confirm this not only as a check but because it will let us compute the constant in Section \ref{s4} when we only know the sum rule up to a constant.

The Verblunsky coefficients for the $\eta$ of \eqref{3.9} are not hard to compute \cite[Example 1.6.4 and equation (1.6.14)]{OPUC1}
\begin{equation} \lb{3.11}
\alpha_n^{(0)} = -\frac{1}{n+2} ; \qquad n=0,1,\dots
\end{equation}
Since  $\sum_{n=0}^{\infty} |\alpha_n^{(0)}|^2 < \infty$, we can cancel the $\tfrac{1}{2}|\alpha_n|^2$ terms in the sums on the right side in \eqref{3.10} and see that when $\alpha=\alpha^{(0)}$ the right side is
\begin{align}
  1-\log(2)+\left(-\frac{1}{2}\right) &- \sum_{n=0}^{\infty} \frac{1}{(n+2)(n+3)} \nonumber \\
   &- \log\left(\prod_{n=0}^{\infty} \left[1-\left(\frac{1}{n+2}\right)^2\right]\right)  \label{3.12}
\end{align}
The sum telescopes since $[(n+2)(n+3)]^{-1}=(n+2)^{-1}-(n+3)^{-1}$ so the sum is $1/2$ and $1-\tfrac{1}{2}-\tfrac{1}{2} = 0$. To evaluate the infinite product, note Euler's formula that
\begin{equation*}
  \sin(\pi x) = \pi x \prod_{j=1}^{\infty} \left(1-\frac{x^2}{j^2}\right)
\end{equation*}
so
\begin{equation*}
  \prod_{n=0}^{\infty} \left(1-\frac{1}{(n+2)^2}\right) = \lim_{x \to 1} \frac{\sin(\pi x)}{\pi (1-x^2)} = -\frac{1}{2\pi}\left. \frac{d}{dx} \sin(\pi x) \right|_{x=1} = \frac{1}{2}
\end{equation*}
and thus the $\log$ term in \eqref{3.12} is $-\log(1/2)$ which cancels the $-\log(2)$.  Thus, we confirm that the expression in \eqref{3.12} is $0$.

The other issue concerns a huge difference in getting sum rules once a $V(\theta)$ is added to the mix.  \corOR{Recall that under
 the $\CUE(N)$ measure, i.e.\ in case $V=0$,
 the measure $\wti{\bbP}_N\in {\mathcal M}_{+,1}(\bbD^{N-1}\times \partial \bbD)$} on
 the Verblunsky coefficients has the property that
if $j < N$, then the Verblunsky coefficients
$(\alpha_0,\dots,\alpha_j)$ are
independent of $(\alpha_{j+1},\dots,\alpha_{N-1})$ so, \corOR{with $\pi_j$
  denoting the continuous projection from
  $\{\alpha_k\}_{k=0}^\infty$ to $\{\alpha_k\}_{k=0}^j$,
the rate function $I_j$ of
$\pi_j^*(\wti{\bbP}_N)\in {\mathcal M}_{+,1}(\bbD^{j+1})$}
is easy to compute (see \cite[Section 2]{BSZ1} for a discussion of $\pi_j^*$). Since $V(U)$ has cross terms between $\alpha_k$ and $\alpha_\ell$  for suitable $k \le j$ and $\ell>j$ (in \eqref{1.6} the $\alpha_{j+1}\alpha_j$ terms), one no longer has independence and the exact calculation of $I_j$ involves the limiting distribution of $\{\alpha_\ell\}_{\ell>j}$.  In the case of \eqref{1.6}, we want to show that $I(\alpha) = F(\alpha_0)+\sum_{k=0}^{\infty} G(\alpha_k,\alpha_{k+1})$ where $G$ has a $\tfrac{1}{2}|\alpha_{k+1}-\alpha_k|^2$ piece and a piece from the $\log(1-|\alpha_k|^2)+|\alpha_k|^2$ term.  Instead of computing $I_j$ exactly, we'll show that (up to constants) $|I_j(\alpha_0,\dots,\alpha_{j-1})-\sum_{k=0}^{j-2}G(\alpha_k,\alpha_{k+1})-F(\alpha_0)| \le C|\alpha_j|$.  This fact and Rakhmanov's Theorem (see \cite[Chapter 9]{OPUC2}) allow one to prove that $I$ has the required form.

We begin the analysis of the general case with

\begin{theorem} \lb{T3.2} Let $V$ be a Laurent polynomial of degree $d$ and let $U_N$ be an $N \times N$ unitary CMV matrix.  Then there \corO{exist}
  $N$--independent polynomials $F_\pm$ and $G$, $G$ depending on $d+1$ successive $\alpha_j$'s and $\bar{\alpha}_j$'s and $F_\pm$ on $d$ such variables so that
\begin{align}
  \tr(V(U_N)) &= F_-(\alpha_0,\dots,\alpha_{d-1})+F_+(\alpha_{N-d},\dots.\alpha_{N-1}) \nonumber \\
              &    \null \qquad +\sum_{j=0}^{N-1-d} G(\alpha_j,\dots,\alpha_{j+d}) \lb{3.13}
\end{align}
Moreover, $G(0,\dots,0) = 0$.
\end{theorem}

\begin{remarks} 1.  The unitary, $U$, associated to any spectral measure $\mu$ is multiplication by $\lambda$ on $L^2(\partial\bbD,d\mu)$.  To get a matrix related with that spectral measure associated to $(1,0,\dots,0)$, one needs to pick an orthonormal basis $\{e_j\}$ for this $L^2$ space with $e_1$ the function $1$.  \cite[Chapter 4]{OPUC1} discusses two natural bases for which the matrix elements are explicit functions of the $\alpha$'s and $\rho$.
  One choice is
  \corO{the set} of orthonormal polynomials for $\mu$.  This yields the GGT matrix.  The other is to orthonormalize $\{1,z,z^{-1},z^2,z^{-2},\dots\}$ which yields the CMV matrix.  One issue is that for general $\mu$, the orthonormal polynomials may not be a basis so the naive GGT matrix may not be unitary but for $n$--point measures, it is unitary.  The CMV matrix is 5 diagonal while the GGT matrix is a Hessenberg matrix, i.e.\ only one non-vanishing diagonal below the principal diagonal but, in general, all non--vanishing matrix elements above the diagonal.  The proof of this theorem will discuss the explicit form of the CMV matrix and \eqref{9.7} \corO{below}
  the explicit form of the GGT matrix.

  2.  These polynomials have degree at most $2d$.  (The CMV matrix has matrix elements that are products of exactly two, $\alpha$, $\bar{\alpha}$ and $\rho$ so $G$ written in terms of the three variables is of homogeneous degree $2d$ if $\corO{\tr( V}(U)) = \tr(U^d)$ but removing the $\rho$'s produces lower degree terms even in this special case.)

3.  $F_\pm, G$ are not unique.  If H is any function of $d$ successive $\alpha, \bar{\alpha}$ pairs and
\begin{equation} \lb{3.14change}
  \begin{cases}
    & \tilde{F}_-(\alpha_0,\dots,\alpha_{d-1}) = F_-(\alpha_0,\dots,\alpha_{d-1}) + H(\alpha_0,\dots,\alpha_{d-1}) \\
    & \tilde{F}_+(\alpha_{N-d},\dots,\alpha_{N-1}) = F_+(\alpha_{N-d},\dots,\alpha_{N-1})+ H(\alpha_{N-d},\dots,\alpha_{N-1}) \\
    & \tilde{G}(\alpha_0,\dots,\alpha_{d}) = G(\alpha_0,\dots,\alpha_{d}) - H(\alpha_0,\dots,\alpha_{d-1}) + H(\alpha_1,\dots,\alpha_d)
  \end{cases}
\end{equation}
then \eqref{3.13} holds for $(G,F_\pm)$ if and only if it holds for $(\tilde{G},\tilde{F}_\pm)$.
\end{remarks}

\begin{proof} Recall \corJ{(}\corO{see} \cite[Section 4.2]{OPUC1}\corJ{)}
  the $\calL\calM$ representation of the CMV matrix, $\calC$, which we write when $N$ is even.  Define the $2 \times 2$ matrices
\begin{equation}\label{3.14}
  \Theta(\alpha) = \left(
     \begin{array}{cc}
       \bar{\alpha} & \rho \\
        \rho & -\alpha \\
     \end{array}
   \right) \qquad \rho = \sqrt{1-|\alpha|^2}
\end{equation}
Let $\Theta_j \equiv \Theta(\alpha_j)$.  Then
\begin{align}
  \calL &= \Theta_0 \oplus \Theta_2 \oplus \dots \oplus \Theta_{N-2} \lb{3.14A} \\
  \calM &= \bdone \oplus \Theta_1 \oplus \dots \oplus \Theta_{N-3} \oplus \alpha_{N-1}\bdone \lb{3.14B}
\end{align}
($\calL$ is a direct sum of $N/2$ $2 \times 2$ matrices while $\calM$ has $1 \times 1$ matrices at the top and bottom and $(N/2-1)$ $2 \times 2$ in between).  And one has that $\calC$ (i.e. our parametrization of $U$) is given by
\begin{equation}\label{3.15}
  \calC = \calL\calM
\end{equation}

We will also write $\wti{\calL}_j, \, j=0,2,\dots, N$ for $\calL$ with $\Theta_0,\dots,\Theta_{j-2},\Theta_{j+2},\dots,\Theta_{N-2}$ replaced by zero (only $\Theta_j$ remains in the direct sum) and similarly for $\wti{\calM}_j, \, j=-1,1,\dots,{N-1}$ (where $\Theta_{-1}, \Theta_{N-1}$ are $1 \times 1$ matrices.  Thus we have that
\begin{equation}\label{3.16}
  \calL=\wti{\calL}_0 + \wti{\calL}_2 + \dots \wti{\calL}_{N-2}; \qquad \calM=\wti{\calM}_{-1}+\dots+\wti{\calM}_{N-1}
\end{equation}
We note that
\begin{equation}\label{3.16A}
  \wti{\calL}_k\wti{\calM}_\ell=\wti{\calM}_\ell\wti{\calL}_k=0 \textrm{ unless } |\ell-k| = 1
\end{equation}
For $N$ odd, there is a similar representation but now $\calL$ has a $1 \times 1$ matrix at the bottom and $\calM$ only a $1 \times 1$ matrix at the top.

We'll prove the theorem when $V(z) =  z^d$.  For $V(z) = z^{-d}$, the argument is similar since replacing $U$ by $U^*$ just interchanges $\calL$ and $\calM$ and replaces $\alpha_j$ by $\bar{\alpha}_j$ (since $\Theta(\alpha)^* = \Theta(\bar{\alpha})$).  And for $0<k<d$, $z^{\pm k}$ yields polynomials of the same form (since functions of fewer variables can be viewed as having more variables; there will be some lost $G$'s near the bottom but they can be made part of $F_+$).

We'll show first that we have the required function of exact degree $2d$ where it is a polynomial in $\alpha, \bar{\alpha}$ and $\rho$ and then that each $\rho_j$ occurs as an even power so using $\rho_j^2=1-\alpha_j\bar{\alpha}_j$ we get the result without any $\rho$'s.

We write
\begin{equation}\label{3.16B}
  (U^d)_{jj} = \sum_{\substack{k_2,\dots,k_{2d} \\ k_1=k_{2d+1}=j}} \sum_{\substack{n_1\dots,n_d \\ m_1,\dots,m_d}}  \wti{\calL}_{n_1; k_1\corO{,}k_2}
  \wti{\calM}_{m_1; k_2\corO{,}k_3} \cdots \wti{\calM}_{m_d; k_{2d}\corO{,}k_{2d+1}}
\end{equation}
where a symbol like $\wti{\calL}_{n_1; k_1\corO{,}k_2}$ means the
$k_1\corO{,}k_2$ matrix element of the matrix $\wti{\calL}_{n_1}$.
In \eqref{3.16B}, we sum $n_1,\dots,n_d,m_1,\dots,m_d$ from $-1$ to $N-1$
running through even and odd integers respectively and
$k_q \, (q=2,\dots,2d)$ \corO{running}
from $0$ to $N-1$.  The only non--zero terms have $ |k_{2p+1}-n_{p+1}| \le 1, \, |k_{q+1}-k_q| \le 1, \, |k_{2p}-m_p| \le 1, \, |n_r-m_r| \le 1, \, |m_r-n_{r+1}| \le 1$, with further restrictions since, for example, $|k_{2p+1}-n_{p+1}| \le 1$ is actually $k_{2p+1}-n_{p+1} =0 \textrm{ or } 1$ and not $-1$.

This clearly writes $\tr(U^d)$ as a polynomial in $\alpha, \bar{\alpha},\rho$ of homogeneous degree $2d$.  For each $j=0,\dots,N-d-1$, group together all where the smallest index of $\alpha, \bar{\alpha},\rho$ is $j$.  It is easy to see that the resulting sum, call it $G_j(\alpha_j,\bar{\alpha}_j,\rho_j,\dots,\rho_{j+2d-1})$, has $G_j$ independent of $j$ and gives the $G$ terms.
The terms with $\alpha_{-1}$ \corO{(coming from $\Theta_{-1}$, and hence $\alpha_{-1}=1$)}
we put into $F_-$ and those whose smallest $j$ \corO{so that}
$j \ge N-d$ \corO{we put} into $F_+$.  It is easy to see that $F_-$ is $N$--independent and that the $N$--dependence of $F_+$ comes only from translating the indices.  Thus we have proven \eqref{3.13} except we have some $\rho$ dependence.

For each product in \eqref{3.16B}, the $\rho_p$ terms come from increasing some $k_q = p$ to $k_{q+1}=p+1$ or a decrease in the opposite direction and it is only through such $\rho_p$ terms that such an increase or decrease can happen.  Since $k_1=k_{2d+1}=j$ and each step only increases or decreases by a single step, for every $\rho_p$ going in one direction, there must be one going in the other, so an even number in all.

To confirm the assertion that $G(0,\dots,0)=0$, we prove that no term in \eqref{3.16B} can only have $\rho$'s, that is there must be at least one $j$ with $k_j=k_{j+1}$.  For if $k_{j+1}=k_{j}\pm 1$ it is easy to see that either $k_{j+2}=k_{j+1}$ or with the same sign $k_{j+2}=k_{j+1}\pm 1$, that is one can't change direction without an $\alpha$ term.  But to return where one started, one must change direction. \end{proof}

\begin{remark} It is an interesting exercise to use the GGT representation \cite[Section 4.1]{OPUC1} to prove that the $\rho$'s only occurs in even powers and that every term in $G$ has at least one power of $\alpha$ or $\bar{\alpha}$.
\end{remark}

 \corOR{In the next theorem, we use
   $\calM_{+1,\infty}(\partial\bbD)$ to denote the subset of $\calM_{+,1}(\partial
	\bbD)$ consisting of
   the probability measures of infinite support on $\partial \bbD$,
 i.e.,\ not supported on finitely many points.}
\begin{theorem} \lb{T3.3} Let V be a potential of the form \eqref{3.4}
  \corOR{with measure $\eta$ whose support is $\partial \bbD$}
  and let $G$ be given by
  \corO{\eqref{3.13}}.
  Let $I$ be the rate function from \corOR{\eqref{3.5}}
  on the measure side.  Let $\pi_L\corJ{\circ \calV}: \corOR{\calM_{+,1}(\partial\bbD)} \to \bbD^L$ mapping $\mu$ to its first $L$ Verblunsky coefficients.
  \corOR{Let $I_L$ be the rate function corresponding to the LDP for
    $\pi_L^*(\wti{\bbP}_N)$, and write $I_L(\mu)=I_L(\pi_L\circ \calV \mu)$}.
  There is a constant $C$ independent of $L$ and $\mu$ so that if $L>d$ and $\alpha$ is the \corJ{sequence of} Verblunsky coefficients of $\mu$, then for all such $L$ and $\mu\in \calM_{+1,\infty}(\partial\bbD)$,
%
\begin{equation}\label{3.17}
  |I_L(\mu) - \sum_{j=0}^{L-d-1} G(\alpha_j,\alpha_{j+1},\dots,\alpha_{j+d})
	\corOR{+}\sum_{j=0}^{L-1} \log(1-|\alpha_j|^2)| \le C
\end{equation}
 \end{theorem}

 \begin{remarks}
\corJ{1. Recall that $\calV$ is the Verblunsky map taking measures to Verblunsky coefficient sequences, defined in the Introduction. The mapping $\pi_L$ is the projection onto the first $L$ elements.}

2. Recall (see \cite[\corJ{Theorem 2.6 and} Theorem 2.7]{BSZ1}) that
   $\pi_L^*(\corO{\wti{\bbP}_N})$ obeys a LDP with speed $N$ and rate $I_L$ related to $I$ by
 \begin{equation}\label{3.18}
   I(\mu) = \sup_L(I_L(\mu)) \qquad I_L(\mu) = \inf_{\{\nu \ |\,\pi_L(\nu)=\pi_L(\mu)\}} I(\nu)
 \end{equation}
 \end{remarks}

 \begin{proof} By writing the induced measures on Verblunsky coefficients according to Killip--Nenciu \cite{KN} (see Theorem 4.2 of \cite{BSZ1}) and $e^{-N\tr(V(U))}$ according to \corOR{\eqref{3.13},}
 we see that for $W \subset \bbD^L$ and $N > L+d+1$
 \begin{equation}\label{3.19}
   \pi_L^*(\corO{\wti{\bbP}}_N)[W]
   = \frac{\int_{\pi_L^{-1}[W]} H_N(\alpha_0,\dots,\alpha_{N-1}) \,\prod_{j=0}^{N-2}d^2\alpha_j d\theta_{N-1}}{\int_{\bbD^{N-2} \times \partial\bbD} H_N(\alpha_0,\dots,\alpha_{N-1})\,\prod_{j=0}^{N-2}d^2\alpha_j d\theta_{N-1}}
 \end{equation}
 where $\alpha_{N-1} = e^{i\theta_{N-1}}$ and
 \begin{align}
   \log H_N(\alpha_0,\dots,\alpha_{N-1}) &= -NF_-(\alpha_0,\dots,\alpha_{d-1})-NF_+(\alpha_{N-d},\dots,\alpha_{N-1}) \nonumber \\
                                        &\null\qquad -N\sum_{j=0}^{N-1-d}G(\alpha_j,\dots,\alpha_{j+d}) \nonumber \\
                                        &\null\qquad + \sum_{j=0}^{N-2} (N-2-j)\log(1-|\alpha_j|^2) \lb{3.20}
 \end{align}

 For fixed $L$, the function $\widetilde{H}_{N,L}$, obtained by dropping the $F_-$ term and all the $G(\alpha_j,\dots,\alpha_j+d)$ terms where $j=L-d,\dots,L-1$ is a product of a function of $(\alpha_0,\dots,\alpha_{L-1})$ and a function of $(\alpha_L,\dots,\alpha_{N-1})$.  Since $\pi_L^{-1}[W] = W\times \bbD^{N-L-1} \times \partial\bbD$ \corJ{(up to a set of zero $\wti{\bbP}_N$ measure)}, the integrals over $(\alpha_L,\dots,\alpha_{N-1})$  in the numerator and denominator of the modified \eqref{3.19} cancel.

 The modified formula defines a probability measure
 \begin{equation}\label{3.21}
   \wti{\bbP}_{N,L}[W] = \frac{\int_W \corOR{\widetilde{H}_{N,L}}(\alpha_0,\dots,\alpha_{L-1}) \,\prod_{j=0}^{L-1} d^2\alpha_j}{\int_{\bbD^L}
	\corOR{\widetilde{H}_{N,L}}(\alpha_0,\dots,\alpha_{L-1}) \,\prod_{j=0}^{L-1}d^2\alpha_j}
 \end{equation}
 where
 \begin{align}
   \log \corOR{\widetilde{H}_{N,L}}(\alpha_0,\dots,\alpha_{L-1}) &= -N \sum_{j=0}^{L-1-d}G(\alpha_j,\dots,\alpha_{j+d}) \nonumber \\
                                        &\null\qquad + \sum_{j=0}^{L-1} (N-2-j)\log(1-|\alpha_j|^2) \lb{3.22}
 \end{align}
 Since $|\alpha_j| \le 1$ and $G$, \corO{$F_-$ and $F_+$}
 are polynomials, the dropped terms are bounded, so that for some constant, $C_1$,
\begin{equation}\label{3.23}
  e^{-C_1 N} \wti{\bbP}_{N,L}(W) \le
  \pi^*_L
  (\wti{\bbP}_N)(W) \le e^{C_1 N} \wti{\bbP}_{N,L}(W)
\end{equation}

By an elementary argument (see \cite[Theorems 2.1 and 2.2]{BSZ1}), $\wti{\bbP}_{N,L}$ obeys a LDP with speed $N$ and rate function
\begin{equation}\label{3.24}
  \ti{I}_L(\alpha_0,\dots,\alpha_{L-1}) = \sum_{j=0}^{L-d-1}G(\alpha_j,\dots,\alpha_{j+d})-\sum_{j=0}^{L-1}\log(1-|\alpha_j|^2) + c_L
\end{equation}
where $c_L$ is such that $\min_{\alpha_0,\dots,\alpha_{L-1}} \ti{I}_L(\alpha_0,\dots,\alpha_{L-1}) = 0$ (forced by the condition on the function $G$ (different from our $G$ here) in \cite[Theorem 2.2]{BSZ1}).

With $I_L$ given by \eqref{3.18}, we conclude by \eqref{3.23} that
\begin{equation}\label{3.25}
  |I_L(\mu) - \ti{I}_L(\alpha_0(\mu),\dots,\alpha_{L-1}(\mu))| \le C_1
\end{equation}

Taking $\mu_0=\tfrac{d\theta}{2\pi}$ for which $I(\mu_0) = \lim I_L(\mu_0)$ is finite and using $G(0,\dots,0)=0$ and $\log(1-|\alpha|^2)|_{\alpha=0} = 0$, we conclude that $c _L$ is bounded as $L \to \infty$ so $\sup c_L \equiv C_2$ is finite.  \eqref{3.17} follows with $C = C_1+C_2$.
\end{proof}

While not essential, the following lovely lemma of Nazarov et.\ al \cite[Lemma 3.1]{NPVY} will simplify some arguments.

\begin{proposition} \lb{P3.4} Let $G$ be a continuous function on $\Omega^k$ where $\Omega \subset \bbR^m$ is compact. Suppose $0 \in \Omega$ and that $G(0,\dots,0)=0$. Let $\Omega^\infty_0$ be the sequences  $\textbf{x}=(x_1,x_2,\dots) \in \Omega^\infty$ so that eventually $x_j=0$ $($i.e.\ only finitely many $x_j$ are non-zero$)$.  For $\textbf{x} \in \Omega^\infty_0$ define
\begin{equation}\label{3.25A}
  H(\textbf{x}) = \sum_{j=0}^{\infty} G(x_{j+1},\dots,x_{j+k})
\end{equation}
Suppose there is a $C$ so that for all $\textbf{x} \in \Omega^\infty_0$, $H(\textbf{x}) \ge -C$.  Then, there exist continuous functions $\wti{G}$ on $\Omega^k$ and $\Gamma$ on $\Omega^{k-1}$ so that
\begin{equation}\label{3.26}
  \wti{G} \ge 0
\end{equation}
and
\begin{equation}\label{3.27}
  G(x_1,\dots,x_k) = \wti{G}(x_1,\dots,x_k) + \Gamma(x_1,\dots,x_{k-1}) - \Gamma(x_2,\dots,x_k)
\end{equation}
\end{proposition}

\begin{remark} The point, of course, is that if we add a constant to $\Gamma$ so that $\Gamma(0,\dots,0)=0$ (which doesn't change \eqref{3.27}), then
\begin{equation*}
  H(\emph{\textbf{x}}) = \Gamma(x_1,\dots,x_{k-1})+\sum_{j=0}^{\infty} \wti{G}(x_{j+1},\dots,x_{j+k})
\end{equation*}
which assures that we can extend $H$ to infinite sequences with a convergent sum or else a sum that diverges to $+\infty$.
\end{remark}

\begin{theorem} [Abstract Gem] \lb{T3.5} Let V be a potential of the form \eqref{3.4} and $G$ given by
  \corO{\eqref{3.13}}.
  Let $\corJ{(\alpha)} \in \bbD^\infty$ and let $\mu=\corJ{\calV^{-1}}({\alpha})$ be the measure with those Verblunsky coefficients and $\eta$ the measure obeying \eqref{3.6}.  Then
\begin{equation}\label{3.28}
  \lim_{N \to \infty} \sum_{j=0}^{N} \left[G(\alpha_j,\dots,\alpha_{j+d}) - \log(1-|\alpha_j|^2)\right]
\end{equation}
exists and the limit is finite if and only if $H(\eta|\mu)$ is finite.
\end{theorem}
\corOR{We refer to the sum in \eqref{3.28} as the \textit{Verblunsky side} of the} \corJR{gem.}
\begin{remark} \cite[Theorem 3.3]{GZ} have a general abstract gem derived by very different means.
\end{remark}


\begin{proof} By the theory of projective limits (see \cite[Theorem 2.7]{BSZ1}), $I(\mu)=\lim_{L \to \infty} I_L(\mu)$.  Thus by \eqref{3.17}, if $I(\mu)=\infty$, the limit in \eqref{3.28} exists and is $\infty$.

  \corO{Assume now that} $I(\mu) < \infty$. \corO{We would like to use
    Proposition \ref{P3.4},
    but first we need to restrict attention to a compact
  subset of the unit disc. Since $I(\mu)<\infty$,}
  the $d\theta$ weight of $d\mu$ is a.e.\ non--zero, so, by Rakhmanov's Theorem (see \cite[Chapter 9]{OPUC2}), $\alpha_j(\mu) \to 0$ as $j \to \infty$.  Thus $R = \sup_j|\alpha_j(\mu)| < 1$.  Let $\overbar{\bbD}_R = \{z\,|\,|z| \le R\}$.  This is compact so we can apply Proposition \ref{P3.4}, \eqref{3.17} and $I(\nu) \ge 0$ for all $\nu$ to conclude that there is $G_R \ge 0$ and $\Gamma_R$ so that
\begin{equation}\label{3.29}
  G(\alpha_0,\dots,\alpha_d) = G_R(\alpha_0,\dots,\alpha_d)+\Gamma(\alpha_0,\dots,\alpha_{d-1})-\Gamma(\alpha_1,\dots,\alpha_d)
\end{equation}
$G(0,\dots,0) = 0 \Rightarrow G_R(0,\dots,0) = 0$ and by adding a constant to $\Gamma$ we can suppose that $\Gamma(0,\dots,0)=0$.

The sum in \eqref{3.28} is thus
\begin{eqnarray}\label{3.30}
  &&\sum_{j=0}^{N} \corOR{[G_R(\alpha_j,\dots,\alpha_{j+d})-\log(1-|\alpha_j|^2)]}
	\\ &&\quad\quad+ \Gamma(\alpha_0,\dots,\alpha_{d-1})-\Gamma(\alpha_{N-1},\dots,\alpha_{N+d})
\nonumber\end{eqnarray}
Since $\alpha_j \to 0, \Gamma(0,\dots,0)=0$ and $\Gamma$ is continuous, the last term goes to $0$ as $N \to \infty$. Since $G_R \ge 0$, the sum has a limit (which may be $+\infty$.  By \eqref{3.17} and $I_L(\mu) \to I(\mu) < \infty$, we see that the sum is bounded, hence convergent.
\end{proof}

Finally, we turn to the abstract sum rule.  For any
$\corO{\alpha}\in \bbD^\infty$ define
\begin{equation}\label{3.31}
  S(\corO{\alpha})= F_-(\alpha_0,\dots,\alpha_{d-1})+\textrm{the limit in \eqref{3.28}}
\end{equation}
$S$ may be infinite if the limit is.

\begin{theorem} [Abstract Sum Rule] \lb{T3.6} Under the hypothesis of Theorem \ref{T3.5}, for any $\mu$ with infinite support
\begin{equation}\label{3.32}
  H(\eta|\mu) = S(\corO{\alpha}(\mu)) - S(\corO{\alpha}(\eta))
\end{equation}
\end{theorem}

\begin{remark} Basically, on the basis of \eqref{3.19}, one expects
  \corO{that}
  the rate function is $S(\corO{\alpha}(\mu))+c$ where $c$ is a constant coming from the $N$th root of the denominator in \eqref{3.19}.  Given that $I(\eta)=0$, the constant has to be $c = -S(\corO{\alpha}(\eta))$.
\end{remark}

\begin{proof} We begin with a formula like \eqref{3.21} but with two changes.  First, rather than look at $\pi^*_L(\corO{\wti{\bbP_N}})[W]$ for a single $W$, we look at a ratio
\begin{equation}\label{3.33}
  \frac{\pi^*_L(\corO{\wti{\bbP_N}})[W]}{\pi^*_L(\corO{\wti{\bbP_N}})[W_1]}
\end{equation}
for two open sets $W, W_1$ in $\bbD^L$ so we needn't concern ourselves with the normalization integral over all of $\bbD^L$ but can focus on small sets where we have control over the $\alpha$'s.

Secondly, we don't drop all of the monomials in those $G$ terms for which $\{j_1,\dots,j_d\}$ intersects both $\{0,\dots,L-1\}$ and $\{L,L+1,\dots\}$.  We keep those monomials which only have $\alpha_L,\alpha_{L+1},\dots$.  Thus the dropped terms all have a factor of some  $\alpha_j$ with $j \in \{L-d,L-d+1,\dots,L-1\}$.  What results is that one obtains (still using $\wti{\bbP}_{N,L}$ for the probability with the, now slightly different, dropped terms):
\begin{equation}\label{3.34}
  e^{-c_L(W,W_1)\corOR{N}} \frac{\wti{\bbP}_{N,L}(W)}{\wti{\bbP}_{N,L}(W_{\corJ{1}})} \le \frac{\pi^*_L (\corJ{\wti{P}_N})[W]}{\pi^*_L(\corJ{\wti{P}_N)}[W_1]} \le e^{c_L(W,W_1)\corOR{N}} \frac{\wti{\bbP}_{N,L}(W)}{\wti{\bbP}_{N,L}(W_{\corJ{1}})}
\end{equation}
where
\begin{equation*}
  c_L(W,W_1) = K\left(\sup_{\corO{\alpha} \in W} \sum_{j=L-d}^{L-1} |\alpha_j|+\sup_{\corO{\alpha} \in W_1} \sum_{j=L-d}^{L-1} |\alpha_j|\right)
\end{equation*}
for some constant $K$ because the dropped terms, \corO{being polynomials
that are not of degree zero in all the $\alpha_j$'s,}
are at least linear in some $\alpha_j$.

Note that because of lower semicontinuity of $I_L$, for any $\mu_0$,
\begin{equation*}
  I_L(\mu_0) = \lim_W \inf_{\mu \in W} I_L(\mu)
\end{equation*}
where $W$ runs over all open neighborhoods of $\mu_0$ ordered by inverse inclusion.  Moreover, because $I_L$ is continuous, one has that
\begin{equation*}
 \inf_{\mu \in W} I_L(\mu) = - \lim_{N \to \infty} \frac{1}{N}
\corOR{\log} \bbP_N(W)
\end{equation*}
Thus taking $N \to \infty$ in \eqref{3.34} and shrinking the open sets to two measures, $\mu$ and $\nu$, we get from \eqref{3.31} that
\begin{flalign}
  &|I_L(\mu)-I_L(\nu)-S_L(\corO{\alpha}(\mu))+S_L(\corO{\alpha}(\nu))| \le &\nonumber \\
  &\null \qquad \qquad \qquad  \qquad \qquad K\left(\sum_{j=L-d}^{L-1} |\alpha_j(\mu)| +\sum_{j=L-d}^{L-1} |\alpha_j(\nu)|\right) & \lb{3.35}
\end{flalign}
where $S_L$ is the sum in \eqref{3.28} when the infinite sum is replaced by the sum to $L-1-d$.

When $H(\eta|\mu) = \infty$, we've already proven \eqref{3.32} so suppose $H(\eta|\mu) < \infty$.  Then \corOR{the density of the absolutely continuous part of $\mu$
with respect to $d\theta$}
is a.e.\ non--vanishing, so, by Rakhmanov's Theorem, $\alpha_j(\mu) \to 0$.  Take $\nu = \eta$ so also, $\alpha_j(\eta) \to 0$.  Thus the right side of \eqref{3.35} goes to zero and we find that
\begin{equation*}
  H(\eta|\mu)-H(\eta|\eta) = \textrm{ RHS of }\eqref{3.32}
\end{equation*}
proving \eqref{3.32}.
\end{proof}

\section{The (1,0) Case} \lb{s4}

In this section, we'll consider the case of a single singularity of order 1 and recover the sum rule of Simon \eqref{1.6}.  The calculations are so simple, we need not make the simplifying assumption that $\bar{\alpha}_j=\alpha_j$ that we'll make in the later sections.

The normalized empirical measure is
\begin{equation}\label{4.1}
  d\eta(\theta) = (1-\cos \theta) \frac{d\theta}{2\pi}
\end{equation}
so, by \eqref{3.4} and \eqref{3.7}
\begin{align}
  V(\theta) &= 2 \int (1-\cos \psi) \log|e^{i\theta} - e^{i\psi}| \, \frac{d\psi}{2\pi}  \nonumber \\
            &= -\int (e^{i\psi}+e^{-i\psi}) \log|e^{i\theta} - e^{i\psi}| \, \frac{d\psi}{2\pi}  \nonumber \\
            &= \frac{1}{2}(e^{i\theta}+e^{-i\theta}) = \cos(\theta) \lb{4.2}
\end{align}
and
\begin{equation}\label{4.3}
  \tr(V(U)) = \frac{1}{2} \tr(U+\overbar{U})
\end{equation}

In the CMV basis, $U_{jj}=-\alpha_{j-1}\bar{\alpha}_j$ where $\alpha_{-1} \equiv -1$.  Thus, the Verblunsky side of the sum rule is
\begin{equation} \lb{4.4}
-\frac{1}{2}\sum_{j=0}^{\infty}(\alpha_{j-1}\bar{\alpha}_j+\bar{\alpha}_{j-1}\alpha_j)-\sum_{j=0}^{\infty} \log(1-|\alpha_j|^2)+C
\end{equation}
for a suitable constant, $C$.

In \eqref{4.4}, the sum rule involves limits of finite $N$ objects so here and below, sums should involve finite matrices and finite sums.  But, as we explained above we are interested in the limits of such finite sums.  So we'll write sums up to infinity indicating what one will get after taking $N \to \infty$ at the end of the calculation.

Since
\begin{equation*}
  \frac{1}{2}|\alpha_j-\alpha_{j-1}|^2 = \frac{1}{2}|\alpha_j|^2+\frac{1}{2}|\alpha_{j-1}|^2 - \frac{1}{2}\alpha_{j-1}\bar{\alpha}_j -\frac{1}{2}\bar{\alpha}_{j-1}\alpha_j
\end{equation*}
we can rewrite \eqref{4.4} as (changed $C$)
\begin{equation}\label{4.5}
  \frac{1}{2}\sum_{j=-1}^{\infty} |\alpha_{j+1}-\alpha_j|^2+\sum_{j=0}^{\infty}\left[-\log(1-|\alpha_j|^2)-|\alpha_j|^2\right]+C
\end{equation}
That in this form the constant is $C=\frac{1}{2}-\log(2)$ follows from the requirement that this vanish if $\corO{\alpha}= \corO{\alpha}(\eta)$ and the calculations in Section \ref{s3} that \corO{\eqref{3.10}} is $0$.  Thus, we have a LD proof of \eqref{1.6}.

To get the gem \eqref{1.7}, we need the $M=1$ case of

\begin{proposition} \lb{P4.1} For any $\corO{\alpha}$
\begin{equation}\label{4.6}
  \sum_{j=0}^{\infty} \left[-\log(1-|\alpha_j|^2)-\sum_{m=1}^{M}\frac{|\alpha_j|^{2m}}{m}\right] < \infty
\end{equation}
if and only if
\begin{equation}\label{4.7}
  \sum_{j=0}^{\infty} |\alpha_j|^{2M+2} < \infty
\end{equation}
\end{proposition}

\begin{remark} Since
\begin{equation}\label{4.8}
  -\log(1-|\alpha|^2) = \sum_{m=1}^{\infty} \frac{|\alpha|^{2m}}{m}
\end{equation}
for any $|\alpha|<1$, the summand in \eqref{4.6} is non--negative so the sum either converges or diverges to $+\infty$.
\end{remark}

\begin{proof} By \eqref{4.8}, we have that
\begin{align}
\frac{|\alpha|^{2M+2}}{M+1}  &\le  -\log(1-|\alpha|^2)-\sum_{m=1}^{M}\frac{|\alpha|^{2m}}{m} \lb{4.9} \\
                                                             &\le \frac{|\alpha|^{2M+2}}{M+1}\left(\sum_{j=0}^{\infty}|\alpha|^{2j}\right) \nonumber \\
                                                             &\le 2\frac{|\alpha|^{2M+2}}{M+1} \quad \textrm{if } |\alpha|^2 \le \frac{1}{2} \lb{4.10}
\end{align}

By \eqref{4.9}, we have that \eqref{4.6}$\Rightarrow$\eqref{4.7}.  On the other hand, if \eqref{4.7} holds, then  $|\alpha_j| \to 0$ so, for all large $j$, $|\alpha_j|^2 \le \frac{1}{2}$ so we can apply \eqref{4.10} to the tail of the sum in \eqref{4.6} and conclude that \eqref{4.7}$\Rightarrow$\eqref{4.6}
\end{proof}

We thus have a quick proof of the gem of Simon \cite[Section 2.8]{OPUC1}:

\begin{theorem} \lb{T4.2} \corOR{With $w$ as in \eqref{1.5},}
$\int_0^{2\pi} (1-\cos \theta) \log w(\theta) \tfrac{d\theta}{2\pi} > -\infty$ if and only if
\begin{equation}\label{4.11}
  \sum_{j=0}^{\infty} |\alpha_{j+1}-\alpha_j|^2 + |\alpha_j|^4 < \infty
\end{equation}
\end{theorem}

\section{The (1,1) Case}  \lb{s5}

In terms of \eqref{1.8}, this section will consider gems where the measure side is
\begin{equation}\label{5.1}
  \int (1+\cos\theta)(1-\cos\theta) \log(w(\theta)) \, \frac{d\theta}{2\pi}
\end{equation}
To figure out the normalization, we note that
\begin{align}
  (1+\cos\theta)(1-\cos\theta) &= 1-\cos^2\theta = \sin^2\theta \nonumber \\
                               &= -\frac{1}{4}(e^{i\theta}-e^{-i\theta})^2 = \frac{1}{2}(1-\cos 2\theta) \lb{5.2}
\end{align}
One can also figure this out by noting that the extreme sides of \eqref{5.2} are degree 2 Laurent polynomials in $e^{i\theta}$ vanishing at $\theta=0,\pi$ to second order with maximum $1$ on $\partial\bbD$. For later use, we note that the same argument shows that for $k=1,2,\dots$
\begin{equation}\label{5.3}
  \prod_{j=0}^{k-1} \left[1-\cos\left(\theta-\frac{2\pi j}{k}\right)\right] = K_k [1-\cos(k\theta)]
\end{equation}
for a constant $K_k$.

Since $\int \cos(k\theta) d\theta = 0$, we see that the normalized $d\eta$ is
\begin{equation} \lb{5.4}
d\eta_k(\theta) = (1-\cos k\theta) \frac{d\theta}{2\pi}
\end{equation}
so by \eqref{3.4} and \eqref{3.7}, we have that
\begin{equation}\label{5.5}
  V_k(\theta) = \frac{1}{k} \cos(k\theta)
\end{equation}
We discussed $k=1$ in Section \ref{s4}, we'll discuss $k=2$ in this section (thereby recovering, using large deviations, a special case of a result of Simon--Zlat\v{o}s \cite{SZ}) and general $k$ in Section \ref{s6.5}.  Thus in this section, we'll prove
\begin{theorem} \lb{T5.1}  Let $\corO{\alpha}$ be real.  Then
\begin{flalign*}
    \int (1-\cos^2\theta) \log(w(\theta)) \, \frac{d\theta}{2\pi} > -\infty \iff &&
\end{flalign*}
 \begin{equation} \lb{5.5a}
   \null \qquad \qquad \qquad \sum_{n=0}^{\infty} |\alpha_{n+2}-\alpha_n|^2 + |\alpha_n|^4 < \infty
\end{equation}
\end{theorem}
\corO{Note that}
$\corO{\tr(V(U))} = \tfrac{1}{4}\tr(U^{2}+U^{-2}) = \tfrac{1}{2} \tr(U^2)$ \corO{if} $\corO{\alpha}$ is real.  For such $\corO{\alpha}$, the CMV matrix has the form for $j \ge 1$ (see \cite[eqn(4.2.14)]{OPUC1}):
\begin{multline}\label{5.6}
   U_{j,j}=- \alpha_j \alpha_{j-1} \qquad U_{2j-1,2j} = -\rho_{2j-1}\alpha_{2j-2} \qquad U_{2j,2j+1} = \corOR{\rho_{2j}} \alpha_{2j+1}\\
   U_{2j,2j-1} = \rho_{2j-1}\alpha_{2j} \qquad U_{2j+1,2j} = -\rho_{2j}\alpha_{2j-1} \qquad \qquad \null
\end{multline}
There are also matrix elements that are two off--diagonal, but if $U_{j,j\pm 2} \ne 0$, then $U_{j\pm 2, j} = 0$, so these terms don't contribute to $\tr(U^2)$ (this is also clear from the $\calL\calM$ factorization and from the GGT representation).  Thus
\begin{align}
  \tr(U^2) &= \textrm{bdy} + \sum_{j=1}^{\infty} [U_{j,j}^2+U_{j,j+1}U_{j+1,j}+U_{j,j-1}U_{j-1,j}] \nonumber \\
           &= \textrm{bdy}+\sum_{j=1}^{\infty} (\alpha_j^2\alpha_{j-1}^2 - 2\rho_j^2\alpha_{j-1}\alpha_{j+1})
					\label{5.8cor}
\end{align}
where $\textrm{bdy}$ \corJ{is short for boundary and} refers to some finite number \corJ{of} terms involving small indices (and, later, when it appears with a finite sum, involving finitely many terms involving large indices with the number of terms bounded as the upper index of the sum changes).

Therefore, using $\rho_j^2 = 1-\alpha_j^2$
\corOR{and $G(\alpha_j,\alpha_{j+1},\alpha_{j+2})=\alpha_{j+1}^2\alpha_j^2-2
\rho_{j+1}^2\alpha_j \alpha_{j+2}$,}
 we see \corOR{after some algebraic manipulations}
that the Verblunsky side of the gem, \corOR{see \eqref{3.28},} is
\begin{equation}\label{5.8}
  \textrm{bdy}+\calI_2+\calI_4+\calL_6
\end{equation}
\begin{align}
  \calI_2 &= \sum_{j=1}^{\infty} (\alpha_j^2 - \alpha_{j-1}\alpha_{j+1}) \lb{5.9}  \\
  \calI_4 &= \sum_{j=1}^{\infty} \left(\tfrac{1}{2} \alpha_j^2\alpha_{j-1}^2 + \alpha_{j-1}\alpha_j^2\alpha_{j+1}+\tfrac{1}{2}\alpha_j^4\right) \lb{5.10} \\
  \calL_6 &= \sum_{j=0}^{\infty} -\log(1-\alpha_j^2) - \alpha_j^2 - \tfrac{1}{2} \alpha_j^4 \lb{5.11}
\end{align}

We claim that up to boundary terms
\begin{equation}\label{5.12}
  \calI_2=\tfrac{1}{2} \sum_{j=1}^{\infty} (\alpha_{j+1}-\alpha_{j-1})^2
\end{equation}
Accepting this for a moment, we can show that the conditions of Simon and Lukic (which agree in this case)
\begin{flalign}\label{5.13}
  \textrm{(S1)} &&\mathclap{\alpha_{j+1}-\alpha_{j-1} \in \ell^2}&&
\end{flalign}
\begin{flalign}\label{5.14}
  \textrm{(S2)} &&\mathclap{\alpha  \in \ell^4}&&
\end{flalign}
imply the measure condition, that is (given the gem) that $\textrm{(S1-2)} \Rightarrow \calI_2 < \infty, \calI_4 < \infty, \calL_6 < \infty$ .  For clearly, \corJ{by $\textrm{(S1)}$}, \eqref{5.12}$< \infty$ and, by H\"{o}lder's inequality, $\calI_4$ is bounded by $C\norm{\alpha}_4^4$.  Since $\norm{\alpha}_6 \le \norm{\alpha}_4$ (on account of $|\alpha|\le 1$), $\calL_6$ is finite by Proposition \ref{P4.1} with $M=2$.

To see \eqref{5.12}, define for $\ell=0,1,2,\dots$
\begin{equation}\label{5.15}
  \calP_\ell = \sum_{j=0}^{\infty} \alpha_j \alpha_{j+\ell}
\end{equation}

\begin{proposition} \lb{P5.2} Let $\calT_1$ and $\calT_2$ be two functions on sequences of real $\alpha$ which are boundary terms plus a sum of the form \eqref{3.25A} where $G$ is a quadratic function of its variables (i.e.\ \corJ{a} second degree homogeneous polynomial).  Suppose for some $L$, $\calT_r$ has no terms of the form $\alpha_j \alpha_{j+\ell}$ with $\ell > L$.  Then up to boundary terms, each $\calT_r$ is a linear combination of $\{\calP_\ell\}_{\ell=0}^L$ and $\calT_1 = \calT_2$ up to boundary terms if they are the same linear combinations.
\end{proposition}

\begin{remarks} 1. This result is obvious.  More subtle is the fact that ``if'' in the last sentence can be replaced by ``if and only if'' but we won't need that harder half of this.

2.  Again, what is being stated involves limits of finite sums.  The equalities only hold up to finite boundary terms.  There are also boundary terms at the upper limit but those go to zero by Rakhmanov's Theorem.  One infinite sum converges if and only if the other one does.
\end{remarks}

\begin{corollary} \lb{C5.3} $\calI_2$ of \eqref{5.9} is given by \eqref{5.12} up to constants.
\end{corollary}

\begin{proof} The RHS of \eqref{5.9} is, up to boundary terms $\calP_0-\calP_2$.  Expanding the square, the RHS of \eqref{5.12} is $\tfrac{1}{2}(\calP_0-2\calP_2+\calP_0) = \calP_0-\calP_2$.
\end{proof}

\begin{proof} [Proof of Theorem 5.1] We've already proven that $\textrm{(S1-2)}$ imply that the integral in \eqref{5.5a} is finite.  So we need to go in the opposite direction.  Therefore, we suppose the integral is finite.

By the abstract gems discussed in Section \ref{s3}, we know that $\calI_2+\calI_4+\calL_6$ is finite (in that the cutoff sums are uniformly bounded) with $\calI_2$ given by \eqref{5.12}.  In this form $\calI_2$ is positive and so is $\calL_6$ as noted in the remark after Proposition \ref{P4.1}.  So we look at $\calI_4$ which we write up to boundary terms as
\begin{align}
  \calI_4    &= \calI_{41}+\calI_{42} \lb{5.16} \\
  \calI_{41} &= \sum_{j=1}^{\infty} \tfrac{1}{4}(\alpha_j\alpha_{j-1}+\alpha_j\alpha_{j+1})^2 \lb{5.17} \\
  \calI_{42} &= \sum_{j=1}^{\infty} \tfrac{1}{2}(\alpha_j^4+\alpha_{j-1}\alpha_j^2\alpha_{j+1}) \lb{5.18}
\end{align}

By H\"{o}lder's inequality,  $\sum_{j=1}^{J} |\alpha_{j-1}\alpha_j^2\alpha_{j+1}| \le \sum_{j=0}^{J+1} |\alpha_j|^4$, so up to boundary terms, $\calI_{42}$ is positive and thus $\calI_2+\calI_{41}+\calL_6$ is finite.  Since each term is positive, they are all finite, i.e.\ $\calI_2 < \infty$ and $\calI_{41} < \infty$.  $\calI_2 < \infty$ is $\textrm{(S1)}$.

$\calI_2 < \infty$ and $|\alpha_j| < 1 \Rightarrow \sum_{j=1}^{\infty} \alpha_j^2(\alpha_{j+1} - \alpha_{j-1})^2 < \infty$.  $\calI_{41}<\infty$ means $\sum_{j=1}^{\infty} \alpha_j^2(\alpha_{j+1} + \alpha_{j-1})^2 < \infty$.  Since $(x+y)^2+(x-y)^2=2(x^2+y^2)$, we conclude that
\begin{equation}\label{5.19}
  \sum_{j=1}^{\infty} \alpha_j^2(\alpha_{j+1}^2+\alpha_{j-1}^2) < \infty
\end{equation}

Since $|\alpha_{j-1}\alpha_j^2\alpha_{j+1}|  \le \tfrac{1}{2} (\alpha_{j-1}^2\alpha_j^2+\alpha_j^2\alpha_{j+1}^2)$, we see that $\sum_{j=1}^{\infty}|\alpha_{j-1}\alpha_j^2\alpha_{j+1}| < \infty$.  All the other terms in $\calI_2+\calI_4+\corJ{\calL}_6$ are positive, so all are finite.  In particular, $\tfrac{1}{2}\sum_{j=1}^{\infty} \alpha_j^4 < \infty$ which is $\textrm{(S2)}$.
\end{proof}

In particular, we see that the proof from Lukic conditions to convergence of the integral is much easier than the converse.

\section{The (2,0) Case} \lb{s6}

Our goal in this section is to prove:

\begin{theorem} \lb{T6.1} \corOR{Let $\mu$ be as in \eqref{1.5} and assume that
its Verblunsky sequence $\alpha$ is real.}  Then
\begin{equation} \lb{6.1}
\int (1-\cos \theta)^2 \log(w(\theta)) \frac{d\theta}{2\pi} > -\infty
\end{equation}
if and only if
\begin{flalign}\label{6.2}
  \textrm{(S1)} &&\mathclap{(S-1)^2\alpha \in \ell^2}&&
\end{flalign}
and
\begin{flalign}\label{6.3}
  \textrm{(S2)} &&\mathclap{\alpha  \in \ell^6}&&
\end{flalign}
\end{theorem}

\begin{remarks} 1. In this case the Lukic and Simon conditions agree.

2.  This result, indeed without the reality restriction is in Simon--Zlat\v{o}s \cite{SZ}.  The main difference in our approach is the method of deriving the sum rules.  Once one has the sum rules the arguments are related but we feel our presentation is more transparent.
\end{remarks}

To begin we need to normalize $\eta$, i.e. determine $Z$ so that $Z^{-1}\int (1-\cos \theta)^2 w(\theta) \frac{d\theta}{2\pi} = 1$.  We'll use
\begin{equation}\label{6.4}
2(1-\cos \theta)=-(e^{i\theta/2}-e^{-i\theta/2})^2
\end{equation}
as one can see by expanding the square or by using $1-\cos \theta = 2 \sin^2(\theta/2)$.  Thus
\begin{align}
   4(1-\cos \theta)^2 &=(e^{i\theta/2}-e^{-i\theta/2})^4 \nonumber \\
                    &=2\cos(2\theta) - 8 \cos(\theta) + 6 \lb{6.5}
\end{align}
Since $\int \cos(k\theta)\tfrac{d\theta}{2\pi} = \delta_{k0}$, we see that
\begin{equation}\label{6.6}
  d\eta(\theta) =\left[1-\frac{4}{3} \cos(\theta) + \frac{1}{3}\cos(2\theta)\right]\,d\theta
\end{equation}
so \corO{by} \eqref{3.4} and \eqref{3.7}
\begin{equation}\label{6.7}
  V(\theta) = -\frac{1}{6} \cos(2\theta) + \frac{4}{3} \cos(\theta)
\end{equation}
and thus when $\corO{\alpha}$ is real
\begin{equation}\label{6.8}
  Q(U) \equiv \tr(V(U)) = -\frac{1}{6} \tr(U^2) + \frac{4}{3} \tr(U)
\end{equation}

We computed $\tr(U)$ in \eqref{4.4} and $\tr(U^2)$ in \corOR{\eqref{5.8cor}}.  Thus the Verblunsky side of the sum rule, \corOR{see \eqref{3.28}}, is
\begin{equation}\label{6.9}
  \textrm{bdy}+\calI_2+\calI_4+\calL_6
\end{equation}
\begin{align}
  \calI_2 &= \sum_{n=1}^{\infty} \left[\alpha_n^2 - \tfrac{4}{3}\alpha_{n+1}\alpha_{n} + \tfrac{1}{3} \alpha_n \alpha_{n+2}\right] \lb{6.10}  \\
  \calI_4 &= \sum_{n=1}^{\infty} \left[\tfrac{1}{2} \alpha_n^4  -\tfrac{1}{6}\alpha_{n-1}^2\alpha_n^2-\tfrac{1}{3}\alpha_{n-1}\alpha_n^2\alpha_{n+1}\right] \lb{6.11} \\
  \calL_6 &= \sum_{n=0}^{\infty} -\log(1-\alpha_n^2) - \alpha_n^2 - \tfrac{1}{2} \alpha_n^4 \lb{6.12}
\end{align}

In terms of the quantities $\calP_\ell$ of \eqref{5.15}
\begin{equation}\label{6.13}
  \calI_2 = \calP_0 - \tfrac{4}{3} \calP_1 +\tfrac{1}{3}\calP_2
\end{equation}
up to boundary terms.  On the other hand, expanding the square, we see that up to boundary terms
\begin{equation}\label{6.14}
  \sum_n (\alpha_{n-1}-2\alpha_{n}+\alpha_{n+1})^2 = 6\calP_0-8\calP_1+2\calP_2
\end{equation}
since $1+4+1=6$ and $4+4=8$.  Thus we see that up to boundary terms
\begin{equation}\label{6.15}
  \calI_2 = \tfrac{1}{6}\sum_n (\alpha_{n-1}-2\alpha_{n}+\alpha_{n+1})^2 = \tfrac{1}{6}\norm{(S-1)^{\corO{2}}\alpha}_2^2
\end{equation}

\begin{proof} [Proof of half of Theorem \ref{T6.1} that \eqref{6.1} $\Rightarrow$ (\corJ{S1}),(S2)]  Since $\tfrac{1}{6}+\tfrac{1}{3} = \tfrac{1}{2}$, up to boundary terms, $\calI_4 \ge 0$ by H\"{o}lder's inequality.  By the abstract sum rule, $\eqref{6.1} \Rightarrow \calI_2+\calI_4+\calI_6$ is finite.  Since each of these terms is positive (by \eqref{6.15} and Proposition \ref{P4.1}), each is individually finite.  $\calI_2 < \infty \Rightarrow$ (S1) and, by Proposition \ref{P4.1}, \corOR{$\calL_6 < \infty \Rightarrow$} (S2). \end{proof}

\begin{proof} [\corOR{Proof of Other Half of Theorem \ref{T6.1} that (S1),(S2)$\Rightarrow $ \eqref{6.1}}]
\corOR{Clearly (S1)$\Rightarrow\calI_2<\infty$ and (S2)
$\corOR{\Rightarrow\calL_6 < \infty}$ by Proposition \ref{P4.1},}
 so we need only control $\calI_4$.  H\"{o}lder
lets one control $\sum \kappa^{(1)}_n  \kappa^{(2)}_n \kappa^{(3)}_n
\kappa^{(4)}_n$ if $\corO{\norm{\kappa^{(j)}}_{p_j}} < \infty$ and $\tfrac{1}{p_1} + \tfrac{1}{p_2} +\tfrac{1}{p_3}+ \tfrac{1}{p_4} \ge 1$.  Since $\tfrac{4}{6} <1$, we can't just look at products of four $\alpha$'s.  However since $\tfrac{1}{6}+\tfrac{1}{6}+\tfrac{1}{6}+\tfrac{1}{2}=1$, we can control products of three $\alpha$'s and one $(S-1)^2\alpha$.  By the Gagliardo-Nirenberg inequality, \eqref{2.11}, (S1)+(S2)$\Rightarrow(S-1)\alpha \in \ell^3$.  Since $\tfrac{1}{6}+\tfrac{1}{6}+\tfrac{1}{3}+\tfrac{1}{3}=1$, a product of two $\alpha$'s and two $(S-1)\alpha$ is also summable.  So the goal is to write $\calI_4$ as sums of these two terms. We write
\begin{align}
  \calI_4 &= \textrm{bdy}+\calI_{41}+\calI_{42} \lb{6.16}  \\
  12\calI_{41} &= 4\sum_{n=1}^{\infty} \alpha_n^2(\alpha_n^2-\alpha_{n+1}\alpha_{n-1}) \lb{6.17} \\
  12\corO{\calI}_{42} &= \sum_{n=0}^{\infty} (\alpha_n^2-\alpha_{n-1}^2)^2 \lb{6.18} \\
             &= \sum_{n=0}^{\infty} (\alpha_n-\alpha_{n-1})^2(\alpha_n+\alpha_{n-1})^2 \lb{6.19}
\end{align}

The $\calI_{42}$ term is a sum of products of two $(S-1)\alpha$ terms and two $\alpha$ terms so by the above, it is a convergent sum by (S1),(S2).  Let $\eta_n=\alpha_{n+1}-\alpha_n$ so
\begin{equation}\label{6.20}
  \alpha_{n+1}\alpha_{n-1} = (\alpha_n+\eta_n)(\alpha_n-\eta_{n-1})
\end{equation}
and thus
\begin{equation}\label{6.21}
  \alpha_n^2(\alpha_n^2-\alpha_{n+1}\alpha_{n-1})=\alpha_n^2\eta_n\eta_{n-1} - \alpha_n^3(\eta_n-\eta_{n-1})
\end{equation}

   We've already seen that the sum in $\calI_{42}$ is absolutely convergent.  By \eqref{6.21}, $\calI_{41}$ is a sum of $(\ell^6)^2(\ell^3)^2$ and $(\ell^6)^3\ell^2$ terms and so a convergent sum.  Thus $\calI_4 < \infty$.
\end{proof}

\section{The $k$th Roots of Unity Case} \lb{s6.5}

Fix $k \in \{1,2,3,\dots\}$.  In this section, we'll consider the conditions
\begin{equation}\label{6.5.1}
  \int (1-\cos k\theta) \log(w(\theta)) \, \frac{d\theta}{2\pi} > -\infty
\end{equation}
\begin{flalign}\label{6.5.2}
  \textrm{(S1)} &&\mathclap{(S^k-1)\alpha \in \ell^2}&&
\end{flalign}
\begin{flalign}\label{6.5.3}
  \textrm{(S2)} &&\mathclap{\alpha  \in \ell^4}&&
\end{flalign}

By \eqref{5.3}, this is the same as taking $\theta_j=\tfrac{2(j-1)\pi}{k}, \, j=1,\dots,k$ so $\{e^{i\theta_j}\}_{j=1}^k$ are the $k$th roots of unity.  Of course, if $\omega=e^{i\theta_2}$ is a primitive $k$th root of unity, then $S^k-1=\prod_{j=1}^{k}(S-\omega^j)$, so \eqref{6.5.2}/\eqref{6.5.3} are precisely the Simon (=Lukic) conditions for this case.  In this section, we'll prove

\begin{theorem} \lb{T6.5.1} Suppose $\corO{\alpha}$ obeys (S2).  Then
\begin{equation*}
  \eqref{6.5.1} \iff \eqref{6.5.2}
\end{equation*}
In particular, (S1-2)$\Rightarrow$\eqref{6.5.1}.
\end{theorem}

\begin{remarks} 1. $\corO{\alpha}$ need not be assumed real.

2. This is a special case of a result of Golinskii--Zlato\v{s} \cite{GZ}.
\end{remarks}

The key input to proving this will be

\begin{proposition} \lb{P6.5.2}  If $\tr(U^k)$ is written in terms of $\alpha$'s only, the term quadratic in $\alpha$ is
\begin{equation}\label{6.5.4}
  \corO{ \textrm{\rm bdy}}-k\sum_{n=0}^{\infty} \alpha_n \bar{\alpha}_{n+k}
\end{equation}
\end{proposition}

\begin{remark} This proof will rely on the CMV representation of unitaries.  It is an interesting exercise to give a different proof using the GGT representation and ideas of Section \ref{s7}.
\end{remark}

\begin{proof} By \eqref{3.16B}, $\tr(U^k)$ is a homogeneous polynomial of degree $2k$ in $\alpha,\bar{\alpha}$ and $\rho$.  To be left with quadratic terms after using $\rho^2 = 1 - \bar{\alpha}\alpha$, we need products with $2k-2 \, \rho$'s and two of $\alpha$ and/or $\bar{\alpha}$.

As the end of the proof of theorem \ref{T3.2} explains, one gets strings of increasing or decreasing $\rho$'s and $\alpha$ or $\bar{\alpha}$ at turn around points.  The $2k-2 \, \rho$'s must occur in a string of $k-1$ increasing and a second string of $k-1$ decreasing $\rho$'s.  The form, \eqref{3.14}, of $\Theta$ shows we get $-\alpha$ at the bottom turn around and $\bar{\alpha}$ at the top turn around, so the only quadratic terms are $(-\alpha_n)\prod_{j=1}^{k-1}\rho_{n+j}(\bar{\alpha}_{n+k})$.

Each diagonal matrix element $(\calC^k)_{jj}$ has such a term for $j=n+1,n+2,\dots,n+k$, so $k$ in all which yields \eqref{6.5.4}.
\end{proof}

\begin{proposition} \lb{P6.5.3} The quadratic term in the sum rule, \eqref{3.24}, for \eqref{6.5.1} with $|\alpha_n|^2$ ``borrowed'' from $-\log(1-|\alpha_n|^2)$ is (up to a boundary term)
\begin{equation} \lb{6.5.5}
\tfrac{1}{2}\sum_{n=0}^{\infty}|\alpha_n-\alpha_{n+k}|^2
\end{equation}
\end{proposition}

\begin{proof} The normalized $\eta$ is $(1-\cos k\theta)\tfrac{d\theta}{2\pi}$, so by \eqref{5.5}, the potential is $\tfrac{1}{k}\cos k\theta$ and $Q$ is
\begin{equation}\label{6.5.6}
  \frac{1}{2k}[\tr(U^k)+\tr(\overbar{U}^k)]
\end{equation}
Thus, since $k$ in \eqref{6.5.4} cancels the $k^{-1}$ in \eqref{6.5.6}, the quadratic term including the borrowed $|\alpha_n|^2$ is
\begin{equation}\label{6.5.7}
  \textrm{bdy}+\tfrac{1}{2}\sum_{n=0}^{\infty} \left[|\alpha_n|^2+|\alpha_{n+k}|^2-\alpha_n\bar{\alpha}_{n+k}-\bar{\alpha}_n\alpha_{n+k}\right]
\end{equation}
which is \eqref{6.5.5}.
\end{proof}

\begin{proof}[\corO{Proof of Proposition \ref{P6.5.3}}]
  The Verblunsky side of the sum rule associated to \eqref{6.5.1} has quadratic term \eqref{6.5.5} and a remainder that is finite if $\corO{\alpha} \in \ell^4$.  Thus the equivalence is immediate.
\end{proof}

\section{Single $k$th Order Singularity} \lb{s6.6}

We are interested here in measures which obey
\begin{equation}\label{6.6.1}
  \int (1-\cos \theta)^k \log(w(\theta)) \, \frac{d\theta}{2\pi} > -\infty
\end{equation}
Here the Simon--Lukic conditions are
\begin{flalign}\label{6.6.2}
  \textrm{(S1)} &&\mathclap{(S-1)^k\alpha \in \ell^2}&&
\end{flalign}
\begin{flalign}\label{6.6.3}
  \textrm{(S2)} &&\mathclap{\alpha  \in \ell^{2k+2}}&&
\end{flalign}
Our main goal is to prove that
\begin{theorem} \lb{T6.6.1} Suppose $\corO{\alpha} \in \ell^4$.  Then
\begin{equation*}
  \eqref{6.6.1} \iff \eqref{6.6.2}
\end{equation*}
\end{theorem}

\begin{remark} This is a special case of a result of Golinskii--Zlato\v{s} \cite{GZ}.
\end{remark}

To put this in perspective, we note that Lukic \cite{Lukic2} has proven
\begin{theorem} [\cite{Lukic2}] Suppose $(S-1)\corO{\alpha} \in \ell^2$.  Then
\begin{equation*}
  \eqref{6.6.1} \iff \eqref{6.6.3}
\end{equation*}
\end{theorem}

These two extreme cases are consistent with $\eqref{6.6.1}\iff$(S1-2) and suggest its truth.

The key to our proof will be to show that the quadratic term in the sum rule is $c_k\norm{(S-1)^k\alpha}_2^2$ for an explicit $c_k$.  We've seen that $c_1=\tfrac{1}{2}$ \eqref{4.5} and $c_2=\tfrac{1}{6}$ \eqref{6.15}.  The reader might stop and try to figure out the general formula.

By \eqref{6.4}
\begin{align}
  2^k(1-\cos\theta)^k &=(-1)^k(e^{i\theta/2}-e^{-i\theta/2})^{2k} \nonumber    \\
                      &= \sum_{j=0}^{2k}\binom{2k}{j}(-1)^{j-k} e^{i(j-k)\theta} \lb{6.6.4}
\end{align}
Thus the normalized $\eta$ is
\begin{equation}\label{6.6.5}
  c_k\sum_{j=0}^{2k}\binom{2k}{j}(-1)^{j-k} e^{i(j-k)\theta} \, \frac{d\theta}{2\pi}
\end{equation}
where
\begin{equation}\label{6.6.6}
  c_k \equiv \frac{1}{\binom{2k}{k}} = \frac{(k!)^2}{(2k)!}
\end{equation}

\corOR{Using the binomial expansion $0=(1-1)^{2k}$, we have that
\[\binom{2k}{k}=-2\sum_{j=k+1}^{2k} \binom{2k}{j}(-1)^{j-k}.\] Therefore,
we may rewrite \eqref{6.6.5} as
\begin{equation}\label{6.6.5bis}
  -2c_k\sum_{j=k+1}^{2k}\binom{2k}{j}(-1)^{j-k} (1-\cos((j-k)\theta))
	\, \frac{d\theta}{2\pi}.
\end{equation}
It follows from  \eqref{5.5} that
\begin{equation}\label{6.6.7}
  V(\theta) = -2c_k\sum_{j=k+1}^{2k}\frac{1}{j-k} \binom{2k}{j}(-1)^{j-k} \cos((j-k)\theta)
\end{equation}
Recalling that we need to borrow $|\alpha_n|^2$ from $-\log(1-|\alpha_n|^2)$, and that the quadratic term in $\tr(U^\ell+\bar{U}^\ell)$ equals
$-\ell \sum_{n=0}^{\infty} (\alpha_n\bar{\alpha}_{n+\ell}+\bar{\alpha}_n\alpha_{n+\ell})$ up to boundary terms,} we see that the quadratic term in the sum rule is
\begin{equation}\label{6.6.8}
 \corOR{
   \calI_2 = \calP_0+2c_k\sum_{j=k+1}^{2k} \binom{2k}{j}(-1)^{j-k}\calP_{j-k}}
\end{equation}
where now, instead of \corJ{\eqref{5.15}}
\begin{equation}\label{6.6.9}
\calP_\ell = \tfrac{1}{2}\sum_{n=0}^{\infty} (\alpha_n\bar{\alpha}_{n+\ell}+\bar{\alpha}_n\alpha_{n+\ell})
\end{equation}
%

On the other hand,
\begin{align}
  \norm{(S-1)^k\alpha}_2^2 &= \sum_{n=0}^{\infty} \left|\sum_{j=0}^{k} \binom{k}{j}(-1)^j \alpha_{n+j}\right|^2 \nonumber  \\
                           &= \corOR{\mbox{\rm bdy}}+
													\corOR{\sum_{j=0}^k  \binom{k}{j}^2 \calP_0}+
													\corOR{2}\sum_{\corOR{\ell=1}}^{k}\left[\sum_{j=0}^{k-\ell} \binom{k}{j} \binom{k}{\ell+j} \right] (-1)^\ell \calP_\ell \lb{6.6.10}
\end{align}
where we use the fact a $\alpha_{n+j_1} \bar{\alpha}_{n+j_2}$ term will contribute to $\calP_\ell$ if $\ell=|j_1-j_2|$.

\begin{proposition} \lb{P6.6.3} For any $k=0,1,2,\dots$ and $\ell=0,\dots,k$, we have that
\begin{equation}\label{6.6.11}
  \sum_{j=0}^{k-\ell} \binom{k}{j} \binom{k}{\ell+j} = \binom{2k}{k+\ell}
\end{equation}
\end{proposition}

\begin{proof} To pick $k-\ell$ elements from among $2k$ numbered objects, we can pick $j$ from the first $k$ and $k-\ell-j$ from the second.  Thus
\begin{equation}\label{6.6.12}
  \sum_{j=0}^{k-\ell}\binom{k}{j}\binom{k}{k-\ell-j} = \binom{2k}{k-\ell}
\end{equation}
Since $\binom{p}{q}=\binom{p}{p-q}$, we have that $\binom{k}{k-\ell-j}=\binom{k}{\ell+j}$ and $\binom{2k}{k-\ell}=\binom{2k}{k+\ell}$. \corJ{We thus} get \eqref{6.6.11}.
\end{proof}

\begin{proof} [Proof of Theorem \ref{T6.6.1}] Picking $j-k=\ell$  in \eqref{6.6.8}, we see that
\begin{equation*}
 \corOR{ \calI_2 = \calP_0+2c_k\sum_{\ell=1}^{k} \binom{2k}{k+\ell}
(-1)^\ell \calP_\ell}
\end{equation*}
which by \eqref{6.6.10},\eqref{6.6.11} and \corOR{\eqref{6.6.6}}
equals $c_k\norm{(S-1)^k\alpha}_2^2$.  When $\alpha \in \ell^4$, by H\"{o}lder's inequality, all terms in the sum rule but the quadratic are finite.  So the Verblunsky side of the sum rule is finite if and only if $\norm{(S-1)^k\alpha}_2^2 <\infty$.  By the sum rule, we conclude the result.
\end{proof}

\section{The (2,1) Case} \lb{s7}

Our \corJ{main result in this section is half} the Lukic conjecture in the $(2,1)$ case, specifically:

\begin{theorem} \lb{T9.1} Let $\mu$ be a probability measure on $\partial\bbD$ of the form \eqref{1.5} with real Verblunsky coefficients $\{\alpha_j\}_{j=0}^\infty$ obeying \eqref{1.15}--\eqref{1.17}.  Then the integral on the left side of \corJ{\eqref{1.11}} is finite.
\end{theorem}

\begin{remark} As noted, this is important because there are examples where Simon's conditions (i.e.\ \eqref{1.15} and \eqref{1.17} without \eqref{1.16}) hold, but the integral in \eqref{1.11} is $-\infty$.
\end{remark}

We'll compute the sum rule guaranteed by Section \ref{s3} to say
$\calI_2+\calI_4+\calI_6+\calL_8 < \infty \iff$ the integral in \eqref{1.12}
is finite, \corO{see \eqref{9.28}-\eqref{9.30} for notation}.  Then we'll show that \eqref{1.15}--\eqref{1.17} $\Rightarrow \calI_2 < \infty, \, \calI_4 < \infty, \, \calI_6 < \infty, \, \calL_8 < \infty$.  We start by computing the potential, $V$, of \eqref{3.4} for the $(2,1)$ case.  As noted (see \eqref{5.2}), we have that
\begin{equation}\label{9.1}
  \cos^2 \theta = \tfrac{1}{4}(e^{i\theta}+e^{-i\theta})^2 = \tfrac{1}{2} \cos 2\theta + \tfrac{1}{2}
\end{equation}
Similarly
\begin{equation}\label{9.2}
  \cos^3 \theta = \tfrac{1}{8}(e^{i\theta}+e^{-i\theta})^3 = \tfrac{1}{4} \cos 3\theta + \tfrac{3}{4}\cos \theta
\end{equation}

Thus
\begin{align}
  P(\theta) &= (1-\cos \theta)^2 (1+\cos \theta)  \nonumber \\
            &= (1-\cos^2 \theta) (1-\cos \theta) \nonumber \\
            &= 1-\cos \theta - \cos^2 \theta  + \cos^3 \theta \lb{9.2A} \\
            &= \tfrac{1}{2} - \tfrac{1}{4} \cos \theta - \tfrac{1}{2} \cos 2\theta + \tfrac{1}{4} \cos 3\theta \lb{9.2B}
\end{align}
by \eqref{9.1}--\eqref{9.2}.  Thus, since $\int \cos k\theta \tfrac{d\theta}{2\pi} = \delta_{k 0}$, the normalized $d\eta$ is
\begin{equation}\label{9.3}
  d\eta(\theta) = \left(1 - \tfrac{1}{2} \cos \theta - \cos 2\theta + \tfrac{1}{2} \cos 3\theta\right)\tfrac{d\theta}{2\pi}
\end{equation}

Using \eqref{3.4} and \eqref{3.7}, we conclude \corJ{that}
\begin{equation}\label{9.4}
  V(\theta) = \tfrac{1}{2} \cos \theta + \tfrac{1}{2} \cos 2\theta - \tfrac{1}{6} \cos 3\theta
\end{equation}
\corJ{so that if $\alpha$ is real then}
\begin{equation}\label{9.5}
  \tr(V(U)) = \tfrac{1}{2} \tr(U) + \tfrac{1}{2} \tr(U^2) - \tfrac{1}{6} \tr(U^3)
\end{equation}

In earlier sections, we used the CMV matrix representation to compute $\tr(U)$ and $\tr(U^2)$.  While initially we computed $\tr(U^3)$ in this way also, we realized the calculations are simpler in the GGT matrix representation.  (GGT and CMV representations are discussed in Section 4.1 and 4.2 of Simon \cite{OPUC1}.)  This is given by
\begin{equation}\label{9.6}
  \calG_{k\ell} = \jap{\varphi_k, z \varphi_\ell}
\end{equation}
The explicit calculation is (Simon \cite[(4.15)]{OPUC1})
\beq \lb{9.7}
\calG_{k\ell}=\left \{ \begin{array}{cc}
-\overline{\alpha}_\ell \alpha_{k-1}\prod_{j=k}^{\ell-1} \rho_{j} & 0 \leq k \leq \ell \\
\rho_\ell & k=\ell+1 \\
0 & k\geq \ell+2
\end{array} \right.
\eeq

In \cite{OPUC1}, this is calculated using $\jap{\Phi_n^*,P} = \norm{\Phi_n}^2 P(0)$ if $\deg P \le n$.  An easier alternative is to use \corO{the}
Szeg\H{o} recursion (\cite[(1.5.25)]{OPUC1}) and inverse Szeg\H{o} recursion (\cite[1.5.46]{OPUC1})
\begin{align}
  z\varphi_n(z) &= \rho_n \varphi_{n+1}(z) + \bar{\alpha}_n \varphi_n^*(z) \lb{9.8} \\
  \varphi_j^*(z) &= \rho_{j-1} \varphi_{j-1}^*(z) - \alpha_{j-1} \varphi_j(z) \lb{9.9}
\end{align}
so
\begin{align}
  z\varphi_n(z) &= \rho_n \varphi_{n+1}(z) - \bar{\alpha}_n \alpha_{n-1} \varphi_n(z) + \bar{\alpha}_n \rho_{n-1} \varphi_{n-1}^*(z) \nonumber \\
                &= \rho_n \varphi_{n+1}(z) - \bar{\alpha}_n \alpha_{n-1} \varphi_n(z) - \bar{\alpha}_n \rho_{n-1} \alpha_{n-2} \varphi_{n-1}(z) \nonumber \\
                &\null \qquad \qquad + \bar{\alpha}_n \rho_{n-1} \rho_{n-2} \varphi_{n-2}^*(z) \lb{9.10}
\end{align}
which upon iterating yields
\begin{equation}\label{9.11}
  z\varphi_n(z) = \rho_n \varphi_{n+1}(z) + \sum_{k=0}^{n} \calG_{kn} \varphi_k(z)
\end{equation}
with $\calG$ given by \eqref{9.7}.

When dealing with the GGT representation, it \corO{can} be an issue that $\{\varphi_n\}_{n=0}^\infty$ is not a basis but the calculations need only be done for finite matrices where the OPs are a basis (or one can use the extended GGT basis of \cite[Section 4.1]{OPUC1} noting that diagonal matrix elements of $\calG^q$ in the extra basis elements are zero).

Define $\calG^{(\ell)}$ to be the $\ell$th diagonal of $\calG$ so ($j,k = 0,\dots,n-1$)
\begin{align}
  \calG^{(-1)}_{jk} &= \rho_k \delta_{k,j-1} \lb{9.12} \\
  \calG^{(\ell)}_{jk} &= -\bar{\alpha}_k \alpha_{j-1} \left[\prod_{m=j}^{k-1} \rho_{m}\right] \delta_{k,j+\ell} \qquad \ell \ge 0 \lb{9.13} \\
  \calG &= \sum_{\ell=-1}^{n-1} \calG^{(\ell)} \lb{9.14}
\end{align}

Of course, only $\calG^{(\ell_1)}\dots\calG^{(\ell_q)}$ with $\sum_{m=1}^{q} \ell_m = 0$ have non--zero main diagonal and so if we expand $\calG^q$ using \eqref{9.14}, only those terms contribute to $\tr(\calG^q)$ so
\begin{equation}\label{9.15}
  \tr(\calG^q) = \sum_{\substack{\ell_1,\dots,\ell_q\\ \sum_{m=1}^{q} \ell_m = 0}}\tr\left(\calG^{(\ell_1)}\dots\calG^{(\ell_q)}\right)
\end{equation}

We can now understand why calculations are easier with the GGT than CMV matrix.  In \eqref{9.15}, the sums start at $\ell_m=-1$ while in the analog for CMV, we start at $\ell_m=-2$, so at least for $q$ not too large, there are fewer terms with GGT.  Moreover, the form of \eqref{9.12}--\eqref{9.13} is covariant under translation along the diagonal while the CMV matrix diagonal\corJ{s} have an even--odd structure.

For $q=1$, we must have $\ell_1=0$ and for $q=2$, we have $(\ell_1,\ell_2)=(0,0),(1,-1)$ or $(-1,1)$.  Moreover, by cyclicity of the trace, the $(1,-1)$ and $(-1,1)$ terms are equal, i.e. $\tr(\calG^2) = \tr\left((\calG^{(0)})^2\right)+2\tr\left(\calG^{(1)}\calG^{(-1)}\right)$.  We thus recover \eqref{4.4} and \eqref{5.8} when $\alpha$ is real, that is up to boundary terms:
\begin{align}
  \tr(\calG) &= I_{1,2} = -\sum_{j=1}^{\infty} \alpha_j \alpha_{j-1} \lb{9.16} \\
  \tr(\calG^2) &= I_{2,2}+I_{2,4} \qquad \qquad I_{2,2} = -2\sum_{j=2}^{\infty}\alpha_j \alpha_{j-2} \lb{9.17} \\
  I_{2,4} &= \sum_{j=2}^{\infty}\alpha_j^2 \alpha_{j-1}^2 + 2 \alpha_j \alpha_{j-1}^2 \alpha_{j-2} \lb{9.18}
\end{align}

For $\tr(\calG^3)$, we have up to cyclic permutations, $(\ell_1,\ell_2,\ell_3)=(0,0,0)$ (once), $(2,-1,-1), (0,1,-1), (0,-1,1)$ (each three times).  Thus up to boundary terms:
\begin{align}
  \tr(\calG^3) &= \tr\left((\calG^{(0)})^3\right)+3\tr\left(\calG^{(2)}(\calG^{(-1)})^2\right)+3\tr\left(\calG^{(0)}\calG^{(1)}\calG^{(-1)}\right) \nonumber \\
  &\null\qquad\qquad +3\tr\left(\calG^{(0)}\calG^{(-1)}\calG^{(1)}\right) \lb{9.19}\\
   &= I_{3,2}+I_{3,4}+I_{3,6} \lb{9.20} \\
   I_{3,2} &= -3\sum_{j=3}^{\infty}\alpha_j \alpha_{j-3} \lb{9.21} \\
   I_{3,4} &= 3\sum_{j=3}^{\infty} \alpha_j(\alpha_{j-1}^2+\alpha_{j-2}^2)\alpha_{j-3} \nonumber \\
   &\null \qquad \qquad + 3\sum_{j=2}^{\infty}\alpha_j\alpha_{j-1}(\alpha_j\alpha_{j-2}+\alpha_{j+1}\alpha_{j-1}) \lb{9.22} \\
   I_{3,6} &= -\sum_{j=1}^{\infty}\alpha_j^3\alpha_{j-1}^3 - 3\sum_{j=3}^{\infty} \alpha_j\alpha_{j-1}^2\alpha_{j-2}^2\alpha_{j-3} \nonumber \\
   &\null\qquad\qquad -3\sum_{j=2}^{\infty} \alpha_j^2\alpha_{j-1}^3\alpha_{j-2} - 3\sum_{j=1}^{\infty} \alpha_j^3\alpha_{j+1}\alpha_{j-1}^2 \lb{9.23}
\end{align}

We also write
\begin{align}
  \calL &\equiv \sum_{j=0}^{\infty}\log(1-\alpha_j^2) = I_{\calL,2} + I_{\calL,4} + I_{\calL,2} + \calL_8 \lb{9.24} \\
  I_{\calL,2} &= \sum_{j=0}^{\infty} \alpha_j^2 \qquad I_{\calL,4}=\tfrac{1}{2}\sum_{j=0}^{\infty} \alpha_j^4 \qquad I_{\calL,6} = \tfrac{1}{3}\sum_{j=0}^{\infty} \alpha_j^6 \lb{9.25} \\
  \calL_8&=-\sum_{j=0}^{\infty} \left[\log(1-\alpha_j^2) - \alpha_j^2 - \tfrac{1}{2} \alpha_j^4 - \tfrac{1}{3} \alpha_j^6 \right] \lb{9.27}
\end{align}
so the \corJ{coefficient side of the} sum rule is $\calI_2+\calI_4+\calI_6+\calL_8$ where
\begin{align}
  \calI_2 &= I_{\calL,2}+\tfrac{1}{2}I_{1,2}+\tfrac{1}{2}I_{2,2}-\tfrac{1}{6}I_{3,2} \lb{9.28} \\
          &= \tfrac{1}{2}\sum_j \left[2\alpha_j^2-\alpha_j\alpha_{j-1}-2\alpha_j\alpha_{j-2}+\alpha_j\alpha_{j-3}\right] \lb{9.29} \\
  \calI_4 &= I_{\calL,4}+\tfrac{1}{2}I_{2,4}-\tfrac{1}{6}I_{3,4} \lb{9.29A} \\
          &= \tfrac{1}{2}\sum_j \left[\alpha_j^4+\alpha_j^2\alpha_{j-1}^2+2\alpha_j\alpha_{j-1}^2\alpha_{j-2}-\alpha_j\alpha_{j-1}^2\alpha_{j-3} \right. \nonumber \\
          &\null\qquad\qquad \left.-\alpha_j\alpha_{j-2}^2\alpha_{j-3} -\alpha_j^2\alpha_{j-1}\alpha_{j-2} - \alpha_j\alpha_{j-1}\alpha_{j-2}^2\right] \lb{9.29B} \\
  \calI_6 &= I_{\calL,6}-\tfrac{1}{6}I_{3,6} \lb{9.29AA} \\
          &= \tfrac{1}{2}\sum_j \left[\tfrac{2}{3}\alpha_j^6+\tfrac{1}{3}\alpha_j^3\alpha_{j-1}^3+\alpha_j\alpha_{j-1}^2\alpha_{j-2}^2\alpha_{j-3} \right. \nonumber \\
          &\null\qquad\qquad \left.+\alpha_j^2\alpha_{j-1}^3\alpha_{j-2} +\alpha_j\alpha_{j-1}^3\alpha_{j-2}^2\right] \lb{9.30}
\end{align}
where we use the fact that adding a constant to all indices in a sum only changes the sum by a boundary term.

We start with $\calI_2$ by using Proposition \ref{P5.2}.  In terms of the $\calP_j$ of \eqref{5.15}, up to boundary terms
\begin{equation}\label{9.31}
  \calI_2 = \tfrac{1}{2}(2\calP_0-\calP_1-2\calP_2+\calP_3)
\end{equation}
by \eqref{9.29}.  On the other hand, by the same calculation that gave \eqref{9.2A}, $(S-1)^2(S+1)=S^3-S^2-S+1$ so
\begin{equation}\label{9.32}
  \left[(S-1)^2(S+1)\alpha\right]_j = \alpha_{j+3} - \alpha_{j+2} - \alpha_{j+1} + \alpha_j
\end{equation}
Thus
\begin{equation}\label{9.33}
  \sum_j \left[(S-1)^2(S+1)\alpha\right]_j^2 = 4\calP_0-2\calP_1-4\calP_2+2\calP_3
\end{equation}
We conclude by Proposition \ref{P5.2} that up to boundary terms
\begin{equation}\label{9.34}
  \calI_2 = \tfrac{1}{4} \norm{(S-1)^2(S+1)\alpha}_2^2
\end{equation}
and thus
\begin{equation}\label{9.35}
  \eqref{1.15} \Rightarrow \calI_2 < \infty
\end{equation}

By H\"{o}lder's inequality
\begin{equation}\label{9.36}
  \eqref{1.17} \Rightarrow \calI_6 < \infty
\end{equation}
By Proposition \ref{P4.1}
\begin{equation}\label{9.36A}
 \eqref{1.17} \Rightarrow \calL_8 < \infty
\end{equation}

Thus, we need to focus on $\calI_4$.  Let $\beta \equiv (S+1)\alpha$.  By Theorem \ref{T2.6} we have that
\begin{equation}\label{9.37}
  \gamma \equiv (S-1)\beta = (S^2-1)\alpha \in \ell^3
\end{equation}
Here is the key first step:

\begin{proposition} \lb{P9.2} \textrm{(a)} For any $m_1,\corJ{m_2,m_3},m_4$, we have that
\begin{equation}\label{9.38}
  \sum_j|\alpha_{j+m_1}\alpha_{j+m_2}\gamma_{j+m_3}\gamma_{j+m_4}| < \infty
\end{equation}

\textrm{(b)} For any $m_1,\corJ{m_2,m_3},m_4$, we have that
\begin{equation}\label{9.39}
  \sum_j|\alpha_{j+m_1}\alpha_{j+m_2}\alpha_{j+m_3}[\gamma_{j+m_4+1}-\gamma_{j+m_4}]| < \infty
\end{equation}

\textrm{(c)} For any $m_1,\corJ{m_2,m_3},m_4$, we have that
\begin{equation}\label{9.40}
  \sum_j \alpha_{j+m_1}\alpha_{j+m_2}\alpha_{j+m_3}\gamma_{j+m_4}
\end{equation}
is conditionally convergent.
\end{proposition}

\begin{remarks} 1. We only need \corJ{conditional} summability so, since $\gamma_j=\alpha_{j+2}-\alpha_j$, (c) implies the conditional summability of the sum in \eqref{9.38} without the $|\cdot |$. However, we use (a) in the proof of (c).

2.  To avoid having to worry about boundary terms at $0$, we extend all sequences to $-\infty$ by setting $\alpha_n=0$ for $n \le -1$.  This doesn't effect conditional convergence of any sums.  Since $\alpha \in \ell^6$, all of $\alpha,\beta,\gamma$ go to zero as $n \to \pm \infty$.
\end{remarks}

\begin{proof} (a) $\tfrac{1}{6}+\tfrac{1}{6}+\tfrac{1}{3}+\tfrac{1}{3}=1$, so since $\alpha \in \ell^6, \, \gamma \in \ell^3$, H\"{o}lder's inequality implies \eqref{9.38}.

(b) $\tfrac{1}{6}+\tfrac{1}{6}+\tfrac{1}{6}+\tfrac{1}{2}=1$, so since $\alpha \in \ell^6, (S-1)\gamma \in \ell^2$, H\"{o}lder's inequality implies \eqref{9.39}.

(c) The intuition is simple.  The continuum analog is that if $f$ is $C^1$ on $\bbR$, $f(x) \to 0$ as $|x| \to \infty$, then $\int_{-R_1}^{R_2} f(x)^3 f'(x) dx = \corJ{\frac{1}{4}}\int_{-R_1}^{R_2} [f^4]'(x) dx$ has a zero limit.  The sum in \eqref{9.40} is a discrete analog so the key will be a summation by parts.

Since we'll be summing by parts, we need to know the appropriate discrete Leibniz rule.  Let $p \in \bbZ\setminus\{0\}$ and $D=S^p-1$ so $(Da)_n=a_{n+p}-a_n$.  Then
\begin{align}
  [D(ab)]_n &= a_{n+p}b_{n+p}-a_nb_n \nonumber \\
            &= a_n(Db)_n+(Da)_n(S^pb)_n \lb{9.41}
\end{align}
or $D(ab)=a(Db)+(Da)S^pb$.  By induction, one sees that
\begin{equation}\label{9.42}
  D(a^{(1)}\dots a^{(k)}) = \sum_{j=1}^{k}a^{(1)}\dots a^{(j-1)} (Da^{(j)})S^pa^{(j+1)}\dots S^pa^{(k)}
\end{equation}

Consider the sum in \eqref{9.40} first if $m_1=m_2=m_3=m_4=0$.  Let $D=S^2-1$.  By \eqref{9.42}
\begin{equation}\label{9.43}
  D(\alpha^4) = (D\alpha)(S^2\alpha)^3+\alpha(D\alpha)(S^2\alpha)^2+\alpha^2(D\alpha)(S^2\alpha)+\alpha^3(D\alpha)
\end{equation}

Given two sequences, $\kappa$ and $\eta$, write $\kappa \stackrel{.}{=} \eta$ to mean $\kappa-\eta \in \ell^1$.  In \eqref{9.43}, $D\alpha=\gamma$ so if we write $S^2\alpha=\alpha+\gamma$, the $\gamma$ term \corJ{produces products} of two $\alpha$'s and two $\gamma$'s, so in $\ell^1$ by (a).  Thus
\begin{equation}\label{9.43A}
  D(\alpha^4) \stackrel{.}{=} 4(D\alpha)\alpha^3 = 4\gamma\alpha^3
\end{equation}

The conditional sum of $D(\alpha^4)$ is finite and indeed zero since $\alpha \in \ell^6$ and
\begin{equation*}
  \sum_{-k}^{n}[D(\alpha^4)]_j = \alpha^4_{n+2} + \alpha^4_{n+1} - \alpha^4_{-k} - \alpha^4_{-k+1} \to 0
\end{equation*}
Thus $4\gamma\alpha^3$ is conditionally summable.

Consider next the case $m_1=m_2=1,m_3=m_4=0$.  By \eqref{9.42} and the same argument that led to \eqref{9.43A}
\begin{align}
  D((\corOR{S}\alpha)^2\alpha^2) &\stackrel{.}{=} 2(S\alpha)^2\alpha(D\alpha)+2\alpha^2(DS\alpha)S\alpha \lb{9.44} \\
                           &\stackrel{.}{=} 2(S\alpha)^2\alpha(D\alpha)+2(S^2\alpha)^2(DS\alpha)S\alpha \lb{9.45}
\end{align}
since, as above, we can replace $\alpha$ by $S^2\alpha$ making an $\ell^1$ error in the four--fold product.

\corOR{Telescoping as in \eqref{9.43A}, we have that $D((S\alpha)^2\alpha^2)$ is conditionally summable.} Note that whether a sequence is conditionally summable or not doesn't change by a translation of index so we can replace $(S^2\alpha)^2(DS\alpha)S\alpha$ by $(S\alpha)^2\alpha(D\alpha)$ and conclude that
\begin{equation}\label{9.47}
  D((\corOR{S}\alpha)^2\alpha^2) - 4(S\alpha)^2\alpha D\alpha
\end{equation}
is conditionally summable and thus $(S\alpha)^2\alpha D\alpha$ is conditionally summable proving the result when $m_1=m_2=1,m_3=m_4=0$.

Now consider general $m_j$.  Since $(S-1)\gamma \in \ell^2$, we can change $m_4$ to any value we want making an $\ell^1$ change.  Similarly, by shifting by multiples of 2 units, we can change each of $m_1, m_2, m_3$ to $0$ or $1$.  If they are all equal after this, set $m_4$ to the common value and get \corJ{conditional} convergence by the case (0,0,0,0).  If the first three $m$'s have two equal and one unequal, set $m_4$ to the unequal value and get either (1,1,0,0) or (0,0,1,1).  We've handled the first and by using the $S^2-1$ trick, (0,0,1,1) is the same as (0,0,-1,-1) and by covariance, that is the same as (1,1,0,0). \end{proof}

Next, we recall the remarkable fact that if \eqref{1.15}+\eqref{1.17}, then $(S-1)\alpha \in \ell^4 \iff (S-1)^2\alpha \in \ell^4$! (see Theorem \ref{T2.6}).

\begin{proof} [Proof of Theorem \ref{T9.1}] As we've seen, we need only show that $\calI_4$ is conditionally convergent.  We only used \eqref{1.15}+\eqref{1.17} so far, but not \eqref{1.16} which we'll use in the form $(S-1)\alpha \in \ell^4$.

We begin by noting that because of (c) of the last Proposition, $\sum \alpha_j^3(\alpha_{j+1}-\alpha_{j-1})$ is conditionally convergent.  Using \corOR{that index shifts modify sums only by boundary terms,}
we conclude that
\begin{equation}\label{9.48}
  \sum_j(\alpha_j^3\alpha_{j-1} - \alpha_j\alpha_{j-1}^3)
\end{equation}
is conditionally convergent.

Since $[\corJ{(}(S-1)\alpha\corJ{)}_{j-1}]^4 = [\alpha_j-\alpha_{j-1}]^4$, using \corOR{again that index shifts do not affect conditional convergence} and \eqref{9.48}, we see that $\norm{(S-1)\alpha}_4^4 < \infty$ implies that
\begin{equation}\label{9.49}
  \sum_j \left[2\alpha_j^4-8\alpha_j^3\alpha_{j-1}+6\alpha_j^2\alpha_{j-1}^2\right]
\end{equation}
is conditionally convergent.

On the other hand, by (c) of the last Proposition, in \eqref{9.29B} we can replace $\alpha_{j-2}$ by $\alpha_j$ and $\alpha_{j-3}$ by $\alpha_{j-1}$ without effecting conditional convergence. If we do that and use \eqref{9.48} again, we see that $\calI_4$ is a conditionally convergent sum plus
\begin{equation}\label{9.50}
  \wti{\calI}_4 = \sum_j \left[\alpha_j^4 + 3\alpha_j^2\alpha_{j-1}^2 - 4\alpha_j^3\alpha_{j-1}\right]
\end{equation}
This is half the sum in \eqref{9.49} so \eqref{1.16} implies conditional convergence of the sum in $\wti{\calI}_4$.
\end{proof}


\end{document}